\documentclass[10pt]{article}
\usepackage[utf8]{inputenc}
\usepackage{diagmac2}
\usepackage{authblk}

\title{Spectra based on\\ Bohl exponents and Bohl dichotomy\\ for nonautonomous difference equations}
\author[1]{Adam Czornik}
\author[2]{Konrad Kitzing}
\author[3]{Stefan Siegmund}
\affil[1]{Faculty of Automatic Control, Electronics and Computer Science, Silesian University of Technology, Gliwice, Poland,
    \texttt{adam.czornik@pols.pl}}
\affil[2]{Institute of Analysis, Faculty of Mathematics, TU Dresden, Germany,
    \texttt{konrad.kitzing@tu-dresden.de}}
\affil[3]{Institute of Analysis, Faculty of Mathematics, TU Dresden, Germany,
    \texttt{stefan.siegmund@tu-dresden.de}}
\date{\today}

\usepackage{cite}

\usepackage{amsmath,amssymb,amsthm,mathtools,bbm}
\providecommand{\definitionname}{Definition}
\providecommand{\theoremname}{Theorem}
\providecommand{\propositionname}{Proposition}
\providecommand{\lemmaname}{Lemma}
\providecommand{\corollaryname}{Corollary}
\providecommand{\remarkname}{Remark}
\providecommand{\examplename}{Example}
\providecommand{\hypothesisname}{Hypothesis}
\theoremstyle{definition}
\newtheorem{definition}{\protect\definitionname}
\theoremstyle{plain}
\newtheorem{theorem}[definition]{\protect\theoremname}

\newtheorem{lemma}[definition]{\protect\lemmaname}

\theoremstyle{remark}
\newtheorem{remark}[definition]{\protect\remarkname}
\theoremstyle{definition}

\theoremstyle{definition}

\numberwithin{equation}{section}
\allowdisplaybreaks
\usepackage{enumitem}
\setlist[enumerate]{label*=(\alph*),ref=(\alph*)}

\newcommand{\N}{\mathbb{N}}

\newcommand{\R}{\mathbb{R}}

\renewcommand{\phi}{\varphi}

\usepackage{xcolor}

\renewcommand{\rho}{\varrho}

\begin{document}

\maketitle

\begin{abstract}
For nonautonomous linear difference equations with bounded coefficients on $\mathbb{N}$ which have a bounded inverse, we introduce two different notions of spectra and discuss their relation to the well-known exponential dichotomy spectrum. The first new spectral notion is called Bohl spectrum and is based on an extended notion of the concept of Bohl exponents. The second new spectral notion is called Bohl dichotomy spectrum and is based on a relaxed version of exponential dichotomy called Bohl dichotomy. We prove spectral theorems and show that the Bohl dichotomy spectrum is the closure of the Bohl spectrum and also a subset of the exponential dichotomy spectrum. We discuss the spectra of upper triangular systems and how they relate to the spectra of their diagonal entries. An example illustrates the subtle differences between the different notions of spectra.

\end{abstract}

\section{Introduction}

Consider the system
\begin{equation}
   x(n+1)=A(n)x(n),
   \quad n \in \mathbb{N}
   \label{1}
\end{equation}
with $A(n)$ in the set $\mathrm{GL}_{d}(\mathbb{R)}$ of invertible $d \times d$ matrices for $n \in  \mathbb{N} = \{0,1,\dots\}$.
We denote the transition matrix of system (\ref{1}) by $\Phi _{A}(n,m)$, $n,$
$m\in \mathbb{N}$, i.e.
\begin{equation*}
   \Phi _{A}(n,m)
   =
   \begin{cases}
      A(n-1)\cdots A(m) & \text{ for } n>m,
   \\
      I_{d} & \text{ for } n=m,
   \\
      \Phi _{A}^{-1}(m,n) & \text{ for } n<m,
   \end{cases}
\end{equation*}
where $I_d$ denotes the identity matrix in $\mathbb{R}^{d \times d}$.
Any solution $\left( x(n)\right) _{n\in \mathbb{N}}$ of \eqref{1} satisfies
\begin{equation*}
   x(n)
   =
   \Phi_{A}(n,m) x(m),
   \quad n, m \in \mathbb{N}.
\end{equation*}
For every $x_0\in \mathbb{R}^{d}$ the unique solution of \eqref{1} which satisfies the initial condition $x(0) = x_0$ is denoted by $x(\cdot,x_0)$. In particular,
\begin{equation*}
   x(n,x_0)
   =
   \Phi_A(n,0) x_0,
   \quad n \in \mathbb{N}.
\end{equation*}
Throughout we assume that $A = (A(n))_{n \in \mathbb{N}}$ and $A^{-1} \coloneqq (A(n)^{-1})_{n \in \mathbb{N}}$ are bounded, i.e.
\begin{align*}
    A \in \mathcal{L}^{\mathrm{Lya}}(\mathbb{N},\mathbb{R}^{d\times d})\coloneqq\{&B\in \mathcal{L}^{\infty}(\mathbb{N},\mathbb{R}^{d\times d}) :
    \\
    &\forall n\in\N : B(n)\in \mathrm{GL}_d(\mathbb R) \text{ and } B^{-1} \in \mathcal{L}^{\infty}(\mathbb{N},\mathbb{R}^{d\times d})\}
\end{align*}
is a so-called \emph{Lyapunov sequence}, where
$\mathcal{L}^{\infty }(\mathbb{N},\mathbb{R}^{d\times d})$ denotes the Banach space of bounded sequences $B = (B(k))_{k \in \mathbb{N}}$ in $\mathbb{R}^{d \times d}$
with norm $\| B \|_{\infty} = \sup_{k \in \mathbb{N}} \| B(k) \|$ and an arbitrary matrix norm $\|\cdot\|$ on $\mathbb{R}^{d \times d}$, see also Remark \ref{rem:PhiBounds}.

A well-studied notion of hyperbolicity for system \eqref{1} is \emph{exponential dichotomy} (see e.g.\ \cite{AulbachSiegmund2001, AulbachSiegmund2002, Barreira2018, Poetzsche2010, Russ2017} and the references therein), which for bounded $A$ and $A^{-1}$ can be defined as follows (cp.\ also \cite[p.\ 2]{Barreira2018}).

\begin{definition}[Exponential dichotomy]\label{ED}
System \eqref{1} has an \emph{exponential dichotomy (ED)} if there exist subspaces $L_1, L_2 \subseteq \mathbb{R}^d$ with $\mathbb{R}^{d}=L_{1} \oplus L_{2}$, $\alpha >0$ and $K > 0$ such
that
\begin{align}
   \| x(n, x_0) \|
   &\leq
   K \mathrm e^{-\alpha (n-m)} \| x(m, x_0) \|,
   & \!\! x_0 \in L_1,
   n \geq m,
   \label{Dich1}
\\
   \| x(n, x_0) \|
   &\geq
   K^{-1} \mathrm e^{\alpha (n-m)} \| x(m, x_0) \|,
   & \!\!  x_0 \in L_2, n \geq m.
   \label{Dich2}
\end{align}
\end{definition}

\begin{remark}[Alternative representation of exponential dichotomy] If system \eqref{1} has an \emph{exponential dichotomy (ED)} and $P \in \mathbb{R}^{d \times d}$ is the projection with $\operatorname{im} P = L_1$ and $\operatorname{ker} P = L_2$ then
\begin{align*}
   \| \Phi_A(n,m) P(m) \|
   &\leq
   K \mathrm e^{-\alpha (n-m)},
   \quad n \geq m,
\\
   \| \Phi_A(m, n) (I-P(n)) \|
   &\leq
   K \mathrm e^{-\alpha (n-m)},
   \quad n \geq m,
\end{align*}
where \(P(n) \coloneqq \Phi_A(n,0) P \Phi_A(0,n)\) is the projection onto \(\Phi_A(n,0)[L_1]\) along \(\Phi_A(n,0)[L_2]\) for \(n \in \N\).
\end{remark}

Rearranging and applying the logarithm, \eqref{Dich1} and \eqref{Dich2} are equivalent to
\begin{align*}
   \frac{1}{n-m} \ln \frac{\| x(n, x_0) \|}{\| x(m, x_0) \|}
   &\leq
   \frac{\ln K}{n-m} - \alpha,
   \quad x_0 \in L_1 \setminus\{0\},
   n > m,
\\
   \frac{1}{n-m} \ln \frac{\| x(n, x_0) \|}{\| x(m, x_0) \|}
   &\geq
   \frac{\ln K^{-1}}{n-m} + \alpha,
   \quad x_0 \in L_2 \setminus\{0\}, n > m.
\end{align*}
These estimates motivate to define the \emph{upper Bohl exponent} and the \emph{lower Bohl exponent} on a subspace $L \subseteq \mathbb{R}^{d}$, $L \neq \{0\}$, by
\begin{align} \label{BohlOnL1}
   \overline{\beta}_A(L)
   &\coloneqq
   \inf_{N \in \mathbb{N}}
   \sup_{n - m > N}
   \sup \Big\{\frac{1}{n-m} \ln \frac{\| x(n, x_0) \|}{\| x(m, x_0) \|} : x_0 \in  L \setminus \{0\} \Big\},
\\ \label{BohlOnL2}
   \underline{\beta}_{A}(L)
   &\coloneqq
   \sup_{N \in \mathbb{N}}
   \inf_{n - m > N}
   \inf \Big\{\frac{1}{n-m} \ln \frac{\| x(n, x_0) \|}{\| x(m, x_0) \|} : x_0 \in  L \setminus \{0\} \Big\},
\end{align}
and $\overline{\beta}_A(\{0\}) \coloneqq -\infty$, $\underline{\beta}_A(\{0\}) \coloneqq +\infty$.
In Section 2 we study these Bohl exponents and their properties as a preparation to define the new notion of Bohl spectrum for equation \eqref{1} based on Bohl exponents
\begin{equation*}
   \Sigma_{B}(A)
   \coloneqq
   \bigcup_{\substack{L \subseteq \mathbb{R}^d\\ \operatorname{dim} L = 1}}
   \big[\underline{\beta}_A(L), \overline{\beta}_A(L)\big]
\end{equation*}
in Section 3. The main result of Section 3 is the Bohl Spectral Theorem \ref{T3} which states that the Bohl spectrum is the non-empty disjoint union of at most $d$ bounded intervals with a corresponding filtration of subspaces consisting of initial values of solutions with corresponding growth rates.
Section 4 is devoted to a new notion of spectrum based on Bohl dichotomy.

The following definition of Bohl dichotomy has been introduced in \cite{Barreira2018} where it is called weak exponential dichotomy.
\begin{definition}[Bohl dichotomy]\label{bohl}
System \eqref{1} has a \emph{Bohl dichotomy (BD)} if there exist subspaces $L_1, L_2 \subseteq \mathbb{R}^d$ with $\mathbb{R}^{d}=L_{1} \oplus L_{2}$, $\alpha >0$ and functions $C_{1}, C_{2} \colon \mathbb{R}^{d}\rightarrow \left( 0,\infty \right) $ such
that
\begin{align}
   \| x(n, x_0) \|
   &\leq
   C_1(x_0) \mathrm e^{-\alpha (n-m)} \| x(m, x_0) \|,
   & \!\! x_0 \in L_1,
   n \geq m,
   \label{11}
\\
   \| x(n, x_0) \|
   &\geq
   C_2(x_0) \mathrm e^{\alpha (n-m)} \| x(m, x_0) \|,
   & \!\! x_0 \in L_2, n \geq m.
   \label{12}
\end{align}
\end{definition}
It is a hyperbolicity notion for \eqref{1} which is similar to exponential dichotomy but weaker in the sense that the constants $C_1$, $C_2$ in \eqref{11}, \eqref{12} are allowed to depend on the solution $x(\cdot,x_0)$ parametrized by $x_0$ in $L_1$ and $L_2$, respectively.
The main result of Section 4 is the Bohl Dichotomy Spectral Theorem \ref{thm:Bohl-Dichotomy-Spectrum} which states that the new notion of Bohl dichotomy spectrum
\begin{equation*}
   \Sigma _{\mathrm{BD}}(A)
   \coloneqq
   \left\{ \gamma \in \mathbb{R}:x(n+1) = \mathrm e^{-\gamma }A(n)x(n)%
    \text{ has no Bohl dichotomy}\right\}
\end{equation*}
is the non-empty disjoint union of at most $d$ compact intervals with a corresponding filtration of subspaces consisting of initial values of solutions with corresponding growth rates.
In Section 5 the new notions of Bohl spectrum and  Bohl dichotomy spectrum are compared with each other and also with the well-known exponential dichotomy spectrum
\begin{equation*}
   \Sigma _{\mathrm{ED}}(A)
   \coloneqq
   \left\{ \gamma \in \mathbb{R} : x(n+1) = \mathrm e^{-\gamma }A(n)x(n)%
    \text{ has no exponential dichotomy} \right\} .
\end{equation*}
In case the linear system \eqref{1} is the linearization of a nonlinear difference equation $x(n+1) = f(n,x(n))$ along a solution $x^*$, i.e.\ $A(n) \coloneqq \frac{\partial f}{\partial x}(n,x^*(n))$, then the stability properties of $x^*$ are related to the spectral properties of its linearization \eqref{1}. This problem and the related theorem of linearized asymptotic stability will be the topic of further research.

\section{Bohl exponents}

A reader who is experienced with characteristic numbers like the Bohl exponents
\begin{align}\label{eq:bohl-recalled1}
   \overline{\beta}_{A}(x_0)
   \coloneqq
   \inf_{N \in \mathbb{N}}
   \sup_{n - m > N} \frac{1}{n-m} \ln \frac{\| x(n, x_0) \|}{\| x(m, x_0) \|},
\\ \label{eq:bohl-recalled2}
   \underline{\beta}_{A}(x_0)
   \coloneqq
   \sup_{N \in \mathbb{N}}
   \inf_{n - m > N} \frac{1}{n-m} \ln \frac{\| x(n, x_0) \|}{\| x(m, x_0) \|},
\end{align}
for $x_0 \in \mathbb{R}^d \setminus\{0\}$, may also be aware of notational and technical challenges when it comes to comparing the existing literature (see also Remark \ref{rem:bohl-history}). The characteristic numbers are often written as a limit superior and limit inferior (see \cite{Doan2017} for a discussion in the continuous time case), respectively, for $n - m \to \infty$
\begin{equation*}
   \overline{\beta}_{A}(x_0)
   =
   \limsup_{n-m \to \infty}
   \tfrac{1}{n-m}
   \ln \tfrac{\| x(n, x_0) \|}{\| x(m, x_0) \|}
   \quad \text{and} \quad
   \underline{\beta}_{A}(x_0)
   =
   \liminf_{n-m \to \infty}
   \tfrac{1}{n-m}
   \ln \tfrac{\| x(n, x_0) \|}{\| x(m, x_0) \|}.
\end{equation*}
Notationally this can be either accepted as an abbreviation of \eqref{eq:bohl-recalled1} and \eqref{eq:bohl-recalled2}, or it can be understood as limit superior and limit inferior \cite[p.\ 217]{Megginson1998}
\begin{align*}
   \limsup_{(n,m) \in D} \lambda(n,m)
   &\coloneqq
   \inf_{(n_0,m_0) \in D}
   \sup \{\lambda(n,m) : (n,m) \geq (n_0,m_0)\}
\\
   \liminf_{(n,m) \in D} \lambda(n,m)
   &\coloneqq
   \sup_{(n_0,m_0) \in D}
   \inf \{\lambda(n,m) : (n,m) \geq (n_0,m_0)\}
\end{align*}
of the real-valued net
\begin{equation*}
   \lambda(n,m) \coloneqq \frac{1}{n-m} \ln \frac{\| x(n, x_0) \|}{\| x(m, x_0) \|},
   \qquad
   (n,m) \in D
\end{equation*}
on the directed set $(D,\leq)$ \cite[Definition 2.1.8]{Megginson1998} with
\begin{equation*}
   D \coloneqq \{(n,m) \in \mathbb{N}^2 \colon n > m\}
\end{equation*}
and preorder $\leq$ on $D$
\begin{equation}\label{eq:preorder}
(n_0, m_0) \leq (n, m)
   \quad :\Leftrightarrow \quad
   n_0 - m_0 \leq n - m.
\end{equation}
This can be seen e.g.\ for $\overline{\beta}_{A}(x_0)$ by using \eqref{eq:bohl-recalled1} and rewriting
\begin{align*}
   \overline{\beta}_{A}(x_0)
   &=
   \inf_{N \in \mathbb{N}}
   \sup \{\lambda(n,m) : n - m > N\}
\\
   &=
   \inf_{(n_0,m_0) \in D}
   \sup \{\lambda(n,m) : n - m \geq n_0 - m_0\}
\\
   &=
   \limsup_{(n,m) \in D} \lambda(n,m).
\end{align*}
The concept of limit superior and limit inferior of a real-valued net also helps to understand an alternative representation of the Bohl exponent which also can be found in the literature (see e.g.\ the monograph \cite[Chapter III]{DaleckiKrein1974} for the continuous time case) and where not only $n - m \to \infty$ but also $m \to \infty$
\begin{equation*}
   \overline{\beta}_{A}(x_0)
   =
   \limsup_{\substack{n-m \to \infty\\ m \to \infty}}
   \tfrac{1}{n-m}
   \ln \tfrac{\| x(n, x_0) \|}{\| x(m, x_0) \|}
   \quad \text{and} \quad
   \underline{\beta}_{A}(x_0)
   =
   \liminf_{\substack{n-m \to \infty\\ m \to \infty}}
   \tfrac{1}{n-m}
   \ln \tfrac{\| x(n, x_0) \|}{\| x(m, x_0) \|}.
\end{equation*}
If the preorder \eqref{eq:preorder} is replaced by
\begin{equation*}
   (n_0, m_0) \leq (n, m)
   \quad :\Leftrightarrow \quad
   n_0 - m_0 \leq n - m
   \;\wedge\;
   n_0 - m_0 \leq m
\end{equation*}
then
\begin{align*}
   \limsup_{(n,m) \in D} \lambda(n,m)
   &=
   \inf_{(n_0,m_0) \in D}
   \sup \{\lambda(n,m) : n - m \geq n_0 - m_0, m \geq n_0 - m_0\}
\\
   &=
   \inf_{N \in \mathbb{N}}
   \sup \{\lambda(n,m) : n - m > N, m > N\}
   =:
   \limsup_{\substack{n-m \to \infty\\ m \to \infty}} \lambda(n,m).
\end{align*}
The following lemma shows that the upper Bohl exponent $\overline{\beta}_A(L)$ in \eqref{BohlOnL1} equals
\begin{align*}
   &
   \limsup_{\substack{n-m\to\infty\\m\to\infty}}\sup_{x_0\in L\setminus\{0\}}\frac 1{n-m}\ln\frac{\Vert x(n,x_0)\Vert}{\Vert x(m,x_0)\Vert}
\\
   &\coloneqq
   \inf_{N \in \mathbb{N}}
   \sup_{\substack{n - m > N \\ m > N}}
   \sup_{x_0 \in  L \setminus \{0\}}\frac{1}{n-m} \ln \frac{\| x(n, x_0) \|}{\| x(m, x_0) \|},
\end{align*}
and the lower Bohl exponent $\underline{\beta}_A(L)$ in \eqref{BohlOnL2} equals
\begin{align*}
   &
   \liminf_{\substack{n-m\to\infty\\m\to\infty}}\inf_{x_0\in L\setminus\{0\}}\frac 1{n-m}\ln\frac{\Vert x(n,x_0)\Vert}{\Vert x(m,x_0)\Vert}
\\
   &\coloneqq
   \sup_{N \in \mathbb{N}}
   \inf_{\substack{n - m > N \\ m > N}}
   \inf_{x_0 \in  L \setminus \{0\}}\frac{1}{n-m} \ln \frac{\| x(n, x_0) \|}{\| x(m, x_0) \|},
\end{align*}
see also \cite{Doan2017}.
Also note, that above and in the definitons of the Bohl exponents \eqref{BohlOnL1} and \eqref{BohlOnL2} the supremum resp.\ infimum over \(N\in\N\) can always be replaced by \(\lim_{N\to\infty}\) by monotonicity.
The fact that $A$ is a Lyapunov sequence plays an important role as pointed out in the next remark.

\begin{remark}[Bounds on transition matrix of Lyapunov sequence]\label{rem:PhiBounds}
Let $n, m \in \mathbb{N}$, $x_0 \in \mathbb{R}^d$. Without referencing, we use the estimates
\begin{equation*}
   \|\Phi_A(n,m)\|
   \leq
   \|A\|_\infty^{n-m}
   \text{ for }
   n \geq m
   \quad \text{and} \quad
   \|\Phi_A(n,m)\|
   \leq
   \|A^{-1}\|_\infty^{n-m}
   \text{ for }
   n \leq m.
\end{equation*}
Moreover, $\Vert A\Vert_\infty \geq 1$ or $\Vert A^{-1}\Vert_\infty \geq 1$, so that $\ln(\max\{\Vert A\Vert_\infty, \Vert A^{-1}\Vert_\infty\}) \geq 0$.
\end{remark}

\begin{lemma}[Alternative representations of Bohl exponents]\label{lem:repr-Bohl-exp}
Let $L \subseteq \mathbb{R}^d$ be a subspace. Then
\begin{align*}
   \overline{\beta}_A(L)
   &=
   \limsup_{\substack{n-m \to \infty \\m \to \infty}}
   \sup_{x_0 \in  L \setminus \{0\}} \frac{1}{n-m} \ln \frac{\| x(n, x_0) \|}{\| x(m, x_0) \|},
\\
   \underline{\beta}_{A}(L)
   &=
   \liminf_{\substack{n-m \to \infty \\m \to \infty}}
   \inf_{x_0 \in  L \setminus \{0\}} \frac{1}{n-m} \ln \frac{\| x(n, x_0) \|}{\| x(m, x_0) \|}.
\end{align*}
\end{lemma}

\begin{proof}
For \(n,m\in\N\) we set
\begin{equation*}
    \overline{\lambda}(n,m)
    \coloneqq
    \sup_{x_0\in L\setminus\{0\}}\frac{1}{n-m}\ln\frac{\Vert x(n,x_0)\Vert}{\Vert x(m,x_0)\Vert}.
\end{equation*}
We prove the statement for $\overline{\beta}_A(L)$. The statement for $\underline{\beta}_{A}(L)$ follows similarly by studying
$\underline{\lambda}(n,m) \coloneqq \inf_{x_0\in L\setminus\{0\}}\frac{1}{n-m}\ln\frac{\Vert x(n,x_0)\Vert}{\Vert x(m,x_0)\Vert}$ instead of $\overline{\lambda}(n,m)$.

For each \(N\in\N\) it holds that
\begin{equation*}
    \sup_{n-m>N}\overline{\lambda}(n,m)
    \geq
    \sup_{\substack{n-m>N\\m>N}}\overline{\lambda}(n,m).
\end{equation*}
Hence
\begin{equation*}
    \inf_{N\in\N}\sup_{n-m>N}\overline{\lambda}(n,m)
    \geq
    \inf_{N\in\N}\sup_{\substack{n-m>N\\m>N}}\overline{\lambda}(n,m).
\end{equation*}
For the converse inequality, we show that for \(C\coloneqq\max\{\Vert A\Vert_\infty, \Vert A^{-1}\Vert_\infty\}\) for each \(N\in\N\), \(N\geq 3\), and  \(n,m\in\N\) with \(n-m>N^2\)
\begin{equation}\label{lem:claim}
    \overline{\lambda}(n,m)
    \leq
    \sup_{\substack{u-w>N\\w>N}}\overline{\lambda}(u,w)+\frac{\ln C}{N}.
\end{equation}
Then with \eqref{lem:claim} it follows for each \(N\in\N\), \(N\geq 3\), that
\begin{equation*}
    \sup_{n-m>N^2}\overline{\lambda}(n,m)
    \leq
    \sup_{\substack{n-m>N\\m>N}}\overline{\lambda}(n,m)+\frac{\ln C}{N}.
\end{equation*}
Then, letting \(N\) tend to infinity and noting that all limits exist
\begin{align*}
    \inf_{N\in\N}
    \sup_{n-m>N}\overline{\lambda}(n,m)
    &=
    \lim_{N\to\infty}
    \sup_{n-m>N^2}\overline{\lambda}(n,m)
\\
    & \leq
    \lim_{N\to\infty}
    \Bigg( \sup_{\substack{n-m>N\\m>N}}\overline{\lambda}(n,m)+\frac{\ln C}{N} \Bigg)
\\
    &=
    \inf_{N\in\N}
    \sup_{\substack{n-m>N\\m>N}}\overline{\lambda}(n,m),
\end{align*}
and the claim follows.

To show \eqref{lem:claim}, let \(N\in\N\), \(N\geq 3\), and \(n,m\in\N\) with \(n-m>N^2\). First we assume that \(m\leq N\).
Then for \(x_0\in L\setminus\{0\}\), noting that \(n-(N+1)>N\) because \(N\geq 3\), we have
\begin{align*}
    \frac{1}{n-m}\ln\frac{\Vert x(n,x_0)\Vert}{\Vert x(m,x_0)\Vert}
    &=
    \frac{1}{n-m}\ln\frac{\Vert x(n,x_0)\Vert}{\Vert \Phi_A(m,N+1)x(N+1,x_0)\Vert}
    \\
    &\leq\frac{1}{n-m}\ln\frac{\Vert x(n,x_0)\Vert\cdot\Vert\Phi_A(N+1,m)\Vert}{\Vert x(N+1,x_0)\Vert}
    \\
    &\leq
    \frac{1}{n-m}\ln\frac{\Vert x(n,x_0)\Vert}{\Vert x(N+1,x_0)\Vert}
    +\frac{1}{n-m}\ln\Vert\Phi_A(N+1,m)\Vert
    \\
    &\leq
    \frac{1}{n-(N+1)}\ln\frac{\Vert x(n,x_0)\Vert}{\Vert x(N+1,x_0)\Vert}
    +\frac{1}{n-m}\ln C^{N-m}
    \\
    &\leq
    \overline{\lambda}(n,N+1)+\frac{N-m}{n-m}\ln C
    \\
    &\leq
    \sup_{\substack{u-w>N\\w>N}}\overline{\lambda}(u,w)+\frac{N}{N^2}\ln C
    \\
    &=
    \sup_{\substack{u-w>N\\w>N}}\overline{\lambda}(u,w)+\frac{\ln C}{N},
\end{align*}
i.e., in case \(m \leq N\), by taking the supremum over \(x_0\in L\setminus\{0\}\)
\begin{equation*}
    \overline{\lambda}(n,m)\leq\sup_{\substack{u-w>N\\w>N}}\overline{\lambda}(u,w)+\frac{\ln C}{N}.
\end{equation*}
In case \(m > N\), note that \(n-m>N^2 \geq N\), and hence also
\begin{align*}
    \overline{\lambda}(n,m)
    &\leq
    \sup_{\substack{u-w>N\\w>N}}\overline{\lambda}(u,w)
\\
    &\leq
    \sup_{\substack{u-w>N\\w>N}}\overline{\lambda}(u,w)+\frac{\ln C}{N},
\end{align*}
proving \eqref{lem:claim}.
\end{proof}

In the following lemma we formulate several properties of Bohl exponents which will be used throughout the paper.

\begin{lemma}[Properties of Bohl exponents]\label{lem:bohlproperties}
Let $L, L_1, L_2$ be subspaces of $\mathbb{R}^d$.
The Bohl exponents defined in \eqref{BohlOnL1} and \eqref{BohlOnL2} satisfy the following properties:

(i) (Bounds) If $L \neq \{0\}$ then
\begin{equation*}
    - \ln \|A^{-1}\|_{\infty} \leq \underline{\beta}_A(L) \leq \overline{\beta}_A(L) \leq \ln \|A\|_{\infty}.
\end{equation*}
Moreover, $\underline{\beta}_A(\{0\}) = \infty$ and $\overline{\beta}_A(\{0\}) = -\infty$.

(ii) (Monotonicity) If $\{0\} \neq L_1 \subseteq L_2$ then
\begin{equation*}
   [\underline{\beta}_A(L_1), \overline{\beta}_A(L_1)] \subseteq [\underline{\beta}_A(L_2), \overline{\beta}_A(L_2)].
\end{equation*}

(iii) (Bohl exponents describe exponential growth on subspaces)
Let \(L\subseteq\R^d\), $L \neq \{0\}$, be a subspace and \(\gamma\in\R\). Then
\begin{equation*}
   \gamma > \overline{\beta}_A(L)
   \quad \Rightarrow \quad
   \begin{matrix}
      \exists K(\gamma) > 0 \;
      \forall x_0 \in L \;
      \forall n > m :
      \\[0.3ex]
      \Vert x(n,x_0) \Vert \leq K \mathrm e^{\gamma(n-m)} \Vert x(m,x_0) \Vert
   \end{matrix}
   \quad \Rightarrow \quad
   \gamma \geq \overline{\beta}_A(L)
\end{equation*}
and
\begin{equation*}
   \gamma < \underline{\beta}_A(L)
   \quad \Rightarrow \quad
   \begin{matrix}
      \exists K(\gamma) > 0 \;
      \forall x_0 \in L \;
      \forall n > m :
      \\[0.3ex]
      \Vert x(n,x_0) \Vert \geq K \mathrm e^{\gamma(n-m)} \Vert x(m,x_0) \Vert
   \end{matrix}
   \quad \Rightarrow \quad
   \gamma \leq \underline{\beta}_A(L).
\end{equation*}

(iv) (Bohl exponents of one-dimensional subspaces)
If $\dim L = 1$ and $x_0 \in L \setminus\{0\}$ then
\begin{equation*}
    \underline\beta_A(x_0)
    =
    \underline\beta_A(L)
    \qquad\text{and}\qquad
    \overline\beta_A(x_0)
    =
    \overline\beta_A(L).
\end{equation*}
In particular, for each $\alpha \in \mathbb{R} \setminus \{0\}$
\begin{equation*}
    \underline\beta_A(x_0)
    =
    \underline\beta_A(\alpha x_0)
    \qquad\text{and}\qquad
    \overline\beta_A(x_0)
    =
    \overline\beta_A(\alpha x_0).
\end{equation*}


(v) (Lower Bohl exponent for exponentially decaying solutions)
Let \(x_0,x_1\in\R^d\), with \(x_0+x_1\in\R^d\setminus\{0\}\).
If \(\overline\beta_A(x_0),\overline\beta_A(x_1)<0\), then \(\underline\beta_A(x_0+x_1)\leq 0\).

(vi) (Lower Bohl exponent for exponentially decaying perturbations)
Let \(x_0\in\R^d\setminus\{0\}\), \(x_1\in\R^d\).
Suppose that \(\underline\beta_A(x_0)>0\) and \(\overline\beta_A(x_1)<0\).
Then
\begin{equation*}
    \underline\beta_A(x_0+x_1)\geq\underline\beta_A(x_0).
\end{equation*}
\end{lemma}

\begin{proof}
(i) To show $\underline{\beta}_A(L) \leq \overline{\beta}_A(L)$, we compute
\begin{align*}
    \underline{\beta}_A(L)
    &=
    \lim_{N \to \infty}
    \inf_{n - m > N}
    \inf_{x_0\in L\setminus\{0\}}\frac{1}{n-m}\ln\frac{\Vert x(n,x_0)\Vert}{\Vert x(m,x_0)\Vert}
    \\
    &\leq
    \lim_{N \to \infty}
    \sup_{n - m > N}
    \sup_{x_0\in L\setminus\{0\}}\frac{1}{n-m}\ln\frac{\Vert x(n,x_0)\Vert}{\Vert x(m,x_0)\Vert}
    =
    \overline{\beta}_A(L).
\end{align*}
To show $\overline{\beta}_A(L) \leq \ln \|A\|_{\infty}$ and $-\ln \|A^{-1}\|_{\infty}\leq\underline{\beta}_A(L)$, we note that
\begin{align*}
    \frac{\Vert x(n,x_0)\Vert}{\Vert x(m,x_0)\Vert}
    &= \frac{\Vert \Phi_A(n,m)x(m,x_0)\Vert}{\Vert x(m,x_0)\Vert}
    \leq \Vert\Phi_A(n,m)\Vert
    \leq \Vert A\Vert_\infty^{n-m},
    \\[1ex]
    \frac{\Vert x(n,x_0)\Vert}{\Vert x(m,x_0)\Vert}
    & = \frac{\Vert x(n,x_0)\Vert}{\Vert \Phi_A(m,n) x(n,x_0)\Vert}
    \geq \Vert\Phi_A(m,n)\Vert^{-1}
    \geq \Vert A^{-1}\Vert_\infty^{-(n-m)}.
\end{align*}

(ii) We prove that $\overline{\beta}_A(L_1) \leq \overline{\beta}_A(L_2)$.
The estimate $\underline{\beta}_A(L_1) \geq \underline{\beta}_A(L_2)$ is shown similarly.
Since \(L_1\subseteq L_2\), it follows for \(m,n\in\N\) with \(n>m\), that
\begin{equation*}
    \sup_{x_0\in L_1 \setminus\{0\}}\frac{1}{n-m}\ln\frac{\Vert x(n,x_0)\Vert}{\Vert x(m,x_0)\Vert}
    \leq
    \sup_{x_0\in L_2 \setminus\{0\}}\frac{1}{n-m}\ln\frac{\Vert x(n,x_0)\Vert}{\Vert x(m,x_0)\Vert},
\end{equation*}
and therefore for each $N \in \mathbb{N}$
\begin{equation*}
    \sup_{n-m>N}
    \sup_{x_0\in L_1 \setminus\{0\}}\frac{1}{n-m}\ln\frac{\Vert x(n,x_0)\Vert}{\Vert x(m,x_0)\Vert}
    \leq
    \sup_{n-m>N}
    \sup_{x_0\in L_2 \setminus\{0\}}\frac{1}{n-m}\ln\frac{\Vert x(n,x_0)\Vert}{\Vert x(m,x_0)\Vert},
\end{equation*}
proving that
\begin{align*}
    \overline\beta(L_1)
    &=
    \lim_{N \to \infty}\sup_{n-m>N}\sup_{x_0\in L_1 \setminus\{0\}}\frac{1}{n-m}\ln\frac{\Vert x(n,x_0)\Vert}{\Vert x(m,x_0)\Vert}
\\
    &\leq
    \lim_{N\to\infty}\sup_{n-m>N}\sup_{x_0\in L_2 \setminus\{0\}}\frac{1}{n-m}\ln\frac{\Vert x(n,x_0)\Vert}{\Vert x(m,x_0)\Vert}
    =
    \overline\beta(L_2).
\end{align*}

(iii) Let \(\gamma>\overline{\beta}_A(L)\). We show that
\begin{equation}\label{eq:exp-est}
   \exists K(\gamma) > 0 \;
      \forall x_0 \in L \;
      \forall n > m :
      \Vert x(n,x_0) \Vert \leq K \mathrm e^{\gamma(n-m)} \Vert x(m,x_0) \Vert
\end{equation}
and then that \eqref{eq:exp-est} implies $\gamma \geq \overline{\beta}_A(L)$. The second statement follows similarly.

Note that
\begin{equation*}
    \overline{\beta}_A(L) = \lim_{N\to\infty}\sup_{\substack{m,n\in\N,\\n-m>N}}\sup\Bigg\{\frac{1}{n-m}\ln\Bigg(\frac{\Vert x(n,x_0)\Vert}{\Vert x(m,x_0)\Vert}\Bigg):x_0\in L\setminus\{0\}\Bigg\}.
\end{equation*}
Hence for \(\varepsilon\coloneqq\gamma - \overline{\beta}_A(L)>0\), there is \(N\in\N\), such that
\begin{equation*}
    \sup_{\substack{m,n\in\N,\\n-m>N}}\sup\Bigg\{\frac{1}{n-m}\ln\Bigg(\frac{\Vert x(n,x_0)\Vert}{\Vert x(m,x_0)\Vert}\Bigg):x_0\in L\setminus\{0\}\Bigg\} - \overline{\beta}_A(L)
    \leq \varepsilon.
\end{equation*}
That is for \(m,n\in\N\), \(n-m>N\) and \(x_0\in L\setminus\{0\}\)
\begin{align*}
    \frac{1}{n-m}\ln\frac{\Vert x(n,x_0)\Vert}{\Vert x(m,x_0)\Vert}
    \leq \varepsilon + \overline{\beta}_A(L)
    = \gamma,
\end{align*}
respectively
\begin{equation*}
    \Vert x(n,x_0)\Vert < \mathrm e^{\gamma(n-m)}\Vert x(m,x_0)\Vert.
\end{equation*}
Now let \(m,n\in\N\) with \(0<n-m\leq N\). Let \(x_0\in L\setminus\{0\}\).
Using the estimates
\begin{gather*}
    \Vert\Phi_A(m+N+1,0)x_0\Vert
    =
    \Vert x(m+N+1,x_0)\Vert
    \leq
    \mathrm e^{\gamma(N+1)}\Vert x(m,x_0)\Vert,
\\
    \|A^{-1}\|_\infty^{-(n-m)}
    \leq
    \max\{1,\|A^{-1}\|_\infty^N \}
    \qquad\text{and}\qquad
    \mathrm{e}^{-\gamma(n-m)}
    \leq
    \max\{1, \mathrm{e}^{-\gamma N}\},
\end{gather*}
we get
\begin{align*}
    \Vert x(n,x_0)\Vert &= \Vert\Phi_A(n,m+N+1)\Phi_A(m+N+1,0)x_0\Vert
    \\
    &\leq \|A^{-1}\|_\infty^{m+N+1-n} \mathrm e^{\gamma(N+1)}\Vert x(m,x_0)\Vert
    \\
    &= \|A^{-1}\|_\infty^{N+1}\mathrm e^{\gamma(N+1)}
    \|A^{-1}\|_\infty^{-(n-m)}
    \mathrm e^{-\gamma(n-m)} \mathrm e^{\gamma(n-m)}
    \Vert x(m,x_0)\Vert
    \\
    &\leq
    K \mathrm e^{\gamma(n-m)}\Vert x(m,x_0)\Vert,
\end{align*}
with $K \coloneqq \|A^{-1}\|_\infty^{N+1} \mathrm e^{\gamma(N+1)} \max\{1,\|A^{-1}\|_\infty^N \} \max\{1, \mathrm{e}^{-\gamma N}\}$.

Suppose now that there is \(K = K(\gamma)\) such that the estimate \eqref{eq:exp-est} holds.
Then for \(x_0\in L\setminus\{0\}\), it follows for \(n>m\) from inequality \eqref{eq:exp-est} that
\begin{equation*}
    \frac{1}{n-m}\ln\frac{\Vert x(n,x_0)\Vert}{\Vert x(m,x_0)\Vert}
    \leq
    \frac{K}{n-m} + \gamma.
\end{equation*}
Hence for all \(N\in\N\), it holds that
\begin{equation*}
    \sup_{n-m>N}\sup_{x_0\in L\setminus\{0\}}\frac{1}{n-m}\ln\frac{\Vert x(n,x_0)\Vert}{\Vert x(m,x_0)\Vert}
    \leq
    \frac{K}{N} + \gamma.
\end{equation*}
Letting \(N\) tend to infinity, acknowledging that all limits exist, it follows that
\begin{equation*}
    \overline\beta_A(L)
    =
    \lim_{N\to\infty}\sup_{n-m>N}\sup_{x_0\in L\setminus\{0\}}\frac{1}{n-m}\ln\frac{\Vert x(n,x_0)\Vert}{\Vert x(m,x_0)\Vert}
    \leq
    \lim_{N\to\infty}\Bigg(\frac{K}{N} + \gamma\Bigg) = \gamma.
\end{equation*}

(iv) This follows from \eqref{BohlOnL1}, \eqref{BohlOnL2}, using the fact that $x(\cdot, \alpha x_0) = \alpha x(\cdot, x_0)$.

(v) The case if \(x_0 = 0\) or \(x_1 = 0\) is clear.
Let \(x_0,x_1\neq 0\). From (iii) we obtain \(\lim\limits_{n\to\infty}x(n,x_0) = \lim\limits_{n\to\infty}x(n,x_1) = 0\).
Hence \(\lim\limits_{n\to\infty}x(n,x_0+x_1) = 0\).
If \(\underline\beta_A(x_0+x_1)>0\), then (iii) would imply that \(\lim\limits_{n\to\infty}x(n,x_0+x_1) = \infty\).
Hence \(\underline\beta_A(x_0+x_1)\leq 0\).

(vi) The case \(x_1 = 0\) is clear.
Let \(x_1\neq 0\).
Let \(\gamma > 0\) with \(\overline\beta_A(x_1)<-\gamma<0\).
By (iii), there is \(K>0\), such that
\begin{equation*}
    \Vert x(n,x_1)\Vert\leq K\mathrm e^{-\gamma(n-m)}\Vert x(m,x_1)\Vert,\quad n>m.
\end{equation*}
Also by (iii), for \(\widehat\gamma\in\R\) with \(0<\widehat\gamma<\underline\beta_A(x_0)\) there is \(\widehat K>0\) with
\begin{equation*}
    \Vert x(n,x_0)\Vert\geq \widehat K\mathrm e^{\widehat\gamma(n-m)}\Vert x(m,x_0)\Vert,\quad n>m.
\end{equation*}
Note that from the previous inequalities it follows that \(\Vert x(n,x_0)\Vert\) resp.\ \(\Vert x(n,x_1)\Vert\) tends to infinity resp.\ zero.
In particular, there is \(N\in\N\), such that
\begin{align*}
    &\widehat K\Vert x(n,x_0)\Vert - K\Vert x(n,x_1)\Vert\geq\frac {\widehat K}2\Vert x(n,x_0)\Vert,
    &n>N,
    \\
    &\Vert x(n,x_0+x_1)\Vert
    \leq \Vert x(n,x_0)\Vert+\Vert x(n,x_1)\Vert
    \leq 2\Vert x(n,x_0)\Vert,
    &n>N.
\end{align*}
We compute for \(m,n\in\N\) with \(m>N\)
\begin{align*}
    \Vert x(n,x_0+x_1)\Vert
    &\geq
    \Vert x(n,x_0)\Vert - \Vert x(n,x_1)\Vert
    \\
    &\geq
    \widehat K\mathrm e^{\widehat\gamma(n-m)}\Vert x(m,x_0)\Vert - K\mathrm e^{-\gamma(n-m)}\Vert x(m,x_1)\Vert
    \\
    &\geq
    \mathrm e^{\widehat\gamma(n-m)}\Big(\widehat K\Vert x(m,x_0)\Vert-K\Vert x(m,x_1)\Vert\Big)
    \\
    &\geq
    \mathrm e^{\widehat\gamma(n-m)}\frac {\widehat K}2\Vert x(m,x_0)\Vert
    \\
    &\geq
    \frac {\widehat K}4\mathrm e^{\widehat\gamma(n-m)}\Vert x(m,x_0+x_1)\Vert.
\end{align*}
Rearranging the inequality and letting \(n-m\) and \(m\) tend to infinity using Lemma \ref{lem:repr-Bohl-exp}, yields \(\underline\beta_A(x_0+x_1)\geq\widehat\gamma\).
The fact that \(\widehat\gamma\in\big(0,\underline\beta_A(x_0)\big)\) was chosen arbitrarily, yields \(\underline\beta_A(x_0+x_1)\geq \underline\beta_A(x_0)\).
\end{proof}

The point of view of Bohl exponents as $\limsup$ and $\liminf$ of the net
\begin{equation*}
   \lambda(n,m) = \frac{1}{n-m} \ln \frac{\| x(n, x_0) \|}{\| x(m, x_0) \|},
   \qquad
   (n,m) \in D
\end{equation*}
on the directed set $(D,\leq)$ is also useful in reinterpreting the following lemma as the statement that every element of a Bohl interval $[\underline{\beta}_{A}(x_0), \overline{\beta}_{A}(x_0)]$ is an accumulation point of the net and can be realized as a limit of a subnet.

\begin{lemma}[Bohl interval as limits of subsequences]\label{BohlInterval-limits}
Let $x_0 \in \mathbb{R}^d \setminus \{0\}$.
Each element in $[\underline{\beta}_{A}(x_0), \overline{\beta}_{A}(x_0)]$ can be realized as a limit, more precisely,
\begin{align*}
   [\underline{\beta}_{A}(x_0), \overline{\beta}_{A}(x_0)]
   &=
   \begin{Bmatrix}
      \lambda \in \mathbb{R} :
      \text{there exist $(n_k)_{k \in \mathbb{N}}, (m_k)_{k \in \mathbb{N}}$ in $\mathbb{N}$ with}
   \\
      \text{$n_k - m_k \to \infty$ and}
   \\
      \text{$\lambda = \lim_{k \to \infty} \tfrac{1}{n_k-m_k} \ln \tfrac{\| x(n_k, x_0) \|}{\| x(m_k, x_0) \|}$}
   \end{Bmatrix}
\\[0.5ex]
   &=
   \begin{Bmatrix}
      \lambda \in \mathbb{R} :
      \text{there exist $(n_k)_{k \in \mathbb{N}}, (m_k)_{k \in \mathbb{N}}$ in $\mathbb{N}$ with}
   \\
      \text{$n_k - m_k \to \infty, m_k \to \infty$ and}
   \\
      \text{$\lambda = \lim_{k \to \infty} \tfrac{1}{n_k-m_k} \ln \tfrac{\| x(n_k, x_0) \|}{\| x(m_k, x_0) \|}$}
   \end{Bmatrix}.
\end{align*}
\end{lemma}

\begin{proof}
Let $x_0 \in \mathbb{R}^d \setminus \{0\}$.
For \(n,m\in\N\) with $n > m$ we set
\begin{equation*}
   \lambda(n,m) \coloneqq \frac{1}{n-m} \ln \frac{\| x(n, x_0) \|}{\| x(m, x_0) \|},
\end{equation*}
denote the second and third set in the equality of Lemma \ref{BohlInterval-limits} by
\begin{align*}
   M_2
   &\coloneqq
   \begin{Bmatrix}
      \lambda \in \mathbb{R} :
      \text{there exist $(n_k)_{k \in \mathbb{N}}, (m_k)_{k \in \mathbb{N}}$ in $\mathbb{N}$ with}
   \\
      \text{$n_k - m_k \to \infty$ and}
   \\
      \text{$\lambda = \lim_{k \to \infty} \lambda(n_k, m_k)$}
   \end{Bmatrix},
\\[0.5ex]
   M_3
   &\coloneqq
   \begin{Bmatrix}
      \lambda \in \mathbb{R} :
      \text{there exist $(n_k)_{k \in \mathbb{N}}, (m_k)_{k \in \mathbb{N}}$ in $\mathbb{N}$ with}
   \\
      \text{$n_k - m_k \to \infty, m_k \to \infty$ and}
   \\
      \text{$\lambda = \lim_{k \to \infty} \lambda(n_k, m_k)$}
   \end{Bmatrix},
\end{align*}
and show that $M_3 \subseteq M_2 \subseteq [\underline{\beta}_{A}(x_0), \overline{\beta}_{A}(x_0)] \subseteq M_3$.

The first inclusion $M_3 \subseteq M_2$ is obvious.

To show that $M_2\subseteq [\underline{\beta}_{A}(x_0), \overline{\beta}_{A}(x_0)]$,
let $\lambda \in M_2$ and $(n_k)_{k \in \mathbb{N}}$, $(m_k)_{k \in \mathbb{N}}$ be sequences in $\mathbb{N}$ with
\begin{equation*}
    n_k>m_k,
    \quad
    \lim\limits_{k\to\infty}(n_k-m_k) = \infty
    \quad
    \text{and}
    \quad
    \lambda = \lim_{k \to \infty}\lambda(n_k,m_k).
\end{equation*}
For $N\in\N$ let $k_N\in\N$ be such that $n_{k_N} - m_{k_N} > N$.
Then
\begin{align*}
    \underline \beta_A(x_0)
    &= \lim_{N\to\infty}\inf_{n-m>N}\lambda(n,m)
    \\
    &\leq \lim_{N\to\infty}\lambda(n_{k_N},m_{k_N})
    \\
    &\leq \lim_{N\to\infty}\sup_{n-m>N}\lambda(n,m)
    = \overline \beta_A(x_0).
\end{align*}
Since $\lim_{N\to\infty}\lambda(n_{k_N},m_{k_N}) = \lambda$, it follows that $\underline \beta_A(x_0)\leq\lambda\leq\overline \beta_A(x_0)$.

To show the inclusion $[\underline{\beta}_{A}(x_0), \overline{\beta}_{A}(x_0)] \subseteq M_3$, let $\lambda \in \big[\underline{\beta}_{A}(x_0), \overline{\beta}_{A}(x_0)\big]$.
The case \(\lambda = \underline{\beta}_{A}(x_0)\) resp.\ \(\lambda = \overline{\beta}_{A}(x_0)\) is clear.
Let \(\lambda\in\big(\underline{\beta}_{A}(x_0), \overline{\beta}_{A}(x_0)\big)\).
Using the representation
\begin{equation*}
   \underline{\beta}_{A}(x_0)
   =
   \lim_{N \to \infty}
   \inf_{\substack{n - m > N \\ m > N}}
   \lambda(n,m),
\end{equation*}
and the fact that $\lambda > \underline{\beta}_{A}(x_0)$, the sequences $(m_\ell)_{\ell \in \mathbb{N}}$, $(n_\ell)_{\ell \in \mathbb{N}}$,
\begin{align*}
   m_\ell
   &\coloneqq
   \min \big\{
   q \in \mathbb{N}
   \;|\;
   \exists p \in \mathbb{N} : (\text{$p - q \geq \ell$ and $q \geq \ell$ and $\lambda(p,q) < \lambda$})
   \big\},
\\
   n_\ell
   &\coloneqq
   \min \big\{
   p \in \mathbb{N}
   \;|\;
   \text{$p - m_\ell \geq \ell$ and $\lambda(p,m_\ell) < \lambda$}
   \big\},
\end{align*}
are well-defined.
Similarly the sequences $(\widetilde{m}_\ell)_{\ell \in \mathbb{N}}$, $(\widetilde{n}_\ell)_{\ell \in \mathbb{N}}$,
\begin{align*}
   \widetilde{m}_\ell
   &\coloneqq
   \min \big\{
   q \in \mathbb{N}
   \;|\;
   \exists p \in \mathbb{N} : (\text{$p - q \geq \ell$ and $q \geq \ell$ and $\lambda(p,q) > \lambda$})
   \big\},
\\
   \widetilde{n}_\ell
   &\coloneqq
   \min \big\{
   p \in \mathbb{N}
   \;|\;
   \text{$p - \widetilde m_\ell \geq \ell$ and $\lambda(p,\widetilde m_\ell) > \lambda$}
   \big\},
\end{align*}
are well-defined.
It holds that
\begin{equation*}
   m_\ell, \widetilde{m}_\ell \geq \ell
   \quad \text{and} \quad
   n_\ell - m_\ell, \widetilde{n}_\ell - \widetilde{m}_\ell \geq \ell,
   \qquad
   \ell \in \mathbb{N}.
\end{equation*}
We show that there exists no $q \in \mathbb{N}$ such that
\begin{equation*}
   m_\ell = \widetilde{m}_\ell = \ell
   \quad \text{and} \quad
   n_\ell - m_\ell = \widetilde{n}_\ell - \widetilde{m}_\ell = \ell,
   \qquad
   \ell \geq q,
\end{equation*}
or, equivalently, $m_\ell = \widetilde{m}_\ell = \ell$ and $n_\ell = \widetilde{n}_\ell = 2\ell$ for $\ell \geq q$.
Assume to the contrary that there exists such a $q \in \mathbb{N}$.
Then
\begin{equation*}
   \lambda(n_\ell, m_\ell) < \lambda
   \quad \text{and} \quad
   \lambda(\widetilde{n}_\ell, \widetilde{m}_\ell) > \lambda,
   \qquad
   \ell \geq q,
\end{equation*}
which is a contradiction because $n_\ell = \widetilde{n}_\ell$ and $m_\ell = \widetilde{m}_\ell$.
As a consequence, there are four cases to consider:

(i) There exists a subsequence $(n_{\ell_k})_{k \in \mathbb{N}}$ of $(n_\ell)_{\ell \in \mathbb{N}}$ with $n_{\ell_k} - m_{\ell_k} > \ell_k$, $k \in \mathbb{N}$.

(ii) There exists a subsequence $(\widetilde n_{\ell_k})_{k \in \mathbb{N}}$ of $(\widetilde n_\ell)_{\ell \in \mathbb{N}}$ with $\widetilde n_{\ell_k} - \widetilde m_{\ell_k} > \ell_k$, $k \in \mathbb{N}$.

(iii) There exists a subsequence $(m_{\ell_k})_{k \in \mathbb{N}}$ of $(m_\ell)_{\ell \in \mathbb{N}}$ with $m_{\ell_k} > \ell_k$, $k \in \mathbb{N}$.

 (iv) There exists a subsequence $(\widetilde m_{\ell_k})_{k \in \mathbb{N}}$ of $(\widetilde m_\ell)_{\ell \in \mathbb{N}}$ with $\widetilde m_{\ell_k} > \ell_k$, $k \in \mathbb{N}$.

We elaborate the details for case (i), the other cases can be discussed in a similar way.
Since $\lambda(n_{\ell_k}, m_{\ell_k}) < \lambda$, it follows that
\begin{equation*}
   \limsup_{k \to \infty}
   \lambda(n_{\ell_k}, m_{\ell_k}) \leq \lambda.
\end{equation*}
We now show
\begin{equation*}
   \liminf_{k \to \infty} \lambda(n_{\ell_k}, m_{\ell_k})
   \geq
   \lambda,
\end{equation*}
proving that $\lim_{k \to \infty} \lambda(n_{\ell_k}, m_{\ell_k}) = \lambda$. To this end, we use the definition of $n_{\ell_k}$ together with the fact that $(n_{\ell_k} - 1) - m_{\ell_k} \geq \ell_k$, $k \in \mathbb{N}$, to conclude that
\begin{equation*}
   \lambda(n_{\ell_k} - 1, m_{\ell_k}) \geq \lambda.
\end{equation*}
Then
\begin{align*}
   &
   \liminf_{k \to \infty} \lambda(n_{\ell_k}, m_{\ell_k}) = {}
\\
   &=
   \liminf_{k \to \infty} \frac{1}{n_{\ell_k} - m_{\ell_k}}
   \ln \frac{\| x(n_{\ell_k}, x_0) \|}{\| x(m_{\ell_k}, x_0) \|}
\\
   &=
   \liminf_{k \to \infty} \frac{1}{n_{\ell_k} - m_{\ell_k}}
   \ln \frac{\|A(n_{\ell_k} - 1)^{-1}\| \cdot \| x(n_{\ell_k}, x_0) \|}
   {\|A(n_{\ell_k} - 1)^{-1}\| \cdot \| x(m_{\ell_k}, x_0) \|}
\\
   &\geq
   \liminf_{k \to \infty} \frac{1}{n_{\ell_k} - m_{\ell_k}}
   \ln \frac{\| x(n_{\ell_k} - 1, x_0) \|}
   {\Vert A^{-1}\Vert_\infty \cdot \| x(m_{\ell_k}, x_0) \|}
\\
   &=
   \liminf_{k \to \infty} \frac{n_{\ell_k} - 1 - m_{\ell_k}}{n_{\ell_k} - m_{\ell_k}}
   \frac{1}{n_{\ell_k} -1 - m_{\ell_k}}
   \ln \frac{\| x(n_{\ell_k} - 1, x_0) \|}
   {\Vert A^{-1}\Vert_\infty \cdot \| x(m_{\ell_k}, x_0) \|}
\\
   &=
   \liminf_{k \to \infty}
   \frac{1}{n_{\ell_k} -1 - m_{\ell_k}}
   \ln \frac{\| x(n_{\ell_k} - 1, x_0) \|}{\| x(m_{\ell_k}, x_0) \|}
\\
   &=
   \liminf_{k \to \infty} \lambda(n_{\ell_k} - 1, m_{\ell_k})
   \geq
   \lambda.
\end{align*}
We denote the subsequences $(n_{\ell_k})_{k\in\N}$, $(m_{\ell_k})_{k\in\N}$ again by $(n_k)_{k\in\N}$, $(m_k)_{k\in\N}$, respectively, to conclude the proof.
\end{proof}

\begin{remark}[Bohl exponents in the literature]\label{rem:bohl-history}
Upper Bohl exponents for solutions and sets of solutions of linear differential equations where introduced by Bohl in his paper \cite{Bohl1913}. In the context of linear difference equations the concept and name of Bohl exponents appeared later, e.g.\ in \cite{Gohberg1991}, where the
quantities
\[
\kappa _{+}\left( L\right) =\underset{j\rightarrow \infty }{\lim }\left(
\underset{n\in \mathbb{N}\text{, }x\in L\backslash \left\{ 0\right\} \text{ }%
}{\sup }\left( \frac{\left\Vert \Phi _{A}\left( n+j,0\right) x\right\Vert }{%
\left\Vert \Phi _{A}\left( n,0\right) x\right\Vert }\right) ^{1/j}\right)
\]%
and%
\[
\kappa _{-}\left( L\right) =\underset{j\rightarrow \infty }{\lim }\left(
\underset{n\in \mathbb{N}\text{, }x\in L\backslash \left\{ 0\right\} \text{ }%
}{\inf }\left( \frac{\left\Vert \Phi _{A}\left( n+j,0\right) x\right\Vert }{%
\left\Vert \Phi _{A}\left( n,0\right) x\right\Vert }\right) ^{1/j}\right) ,
\]%
are introduced for a subspace $L$ of $\mathbb{R}^{d}$ (see also the review \cite{Barabanov2001}). They are related to the upper and lower Bohl exponents \eqref{BohlOnL1}, \eqref{BohlOnL2} by
\[
\ln \kappa _{+}\left( L\right) =\overline{\beta }_{A}\left(L\right)
\text{ and }\ln \kappa _{-}\left( L\right) =\underline{\beta }_{A}\left(L\right).
\]%
For $L = \mathbb{R}^d$ sometimes (see e.g.\ \cite{HinrichsenPritchard2005} and the references therein) a different notation is used for the Bohl exponents $\overline{\beta }_{A}\left(L\right)$ and $\underline{\beta }_{A}\left(L\right)$, respectively.
\begin{equation*}
   \Omega (A)
   =
   \overline{\beta }_{A}\left(\mathbb{R}^d\right)
   \qquad\text{and}\qquad
   \omega (A)
   =
   \underline{\beta }_{A}\left(\mathbb{R}^d\right)
\end{equation*}
are called e.g.\ general exponents \cite{Barabanov2001} or singular exponents \cite{Izobov2012}.
\end{remark}

\section{Bohl spectrum}

We define a notion of spectrum of \eqref{1} based on Bohl intervals formed by Bohl exponents and prove a spectral theorem.
This is a discrete analogue of the Bohl spectrum defined in \cite{Doan2017}.

\begin{definition}[Bohl spectrum]\label{def:Bohl-spectrum}
\label{B}The Bohl spectrum of \eqref{1} is defined as
\begin{equation*}
   \Sigma _{\mathrm{B}}(A)
   \coloneqq
   \bigcup_{x_0 \in \mathbb{R}^d \setminus \{0\}}
   [\underline{\beta}_A(x_0), \overline{\beta}_A(x_0)].
\end{equation*}
Its complement $\varrho_{\mathrm{B}}(A) \coloneqq \mathbb{R} \setminus \Sigma _{\mathrm{B}}(A)$ is called the resolvent of \eqref{1}.
\end{definition}

\begin{remark}[Bohl spectrum is bounded]\label{rem:bohl-spectrum-bounded}
By Lemma \ref{lem:bohlproperties}(i),
\begin{equation*}
    - \ln \|A^{-1}\|_{\infty}
    \leq
    \underline{\beta}_A(x_0) \leq \overline{\beta}_A(x_0) \leq \ln \|A\|_{\infty}
\end{equation*}
for \(x_0\in\mathbb R^d\setminus\{0\}\). Hence
$\Sigma _{\mathrm{B}}(A) \subseteq \big[-\ln \|A^{-1}\|_{\infty}, \ln \|A\|_{\infty}\big]$.
\end{remark}

To prepare the formulation and proof of the Bohl spectral theorem, we introduce a $\gamma$-dependent set of initial conditions with upper Bohl exponent less than $\gamma$. Conceptually this set corresponds to the sum of the generalized eigenspaces of a constant matrix $A$ which correspond to eigenvalues with a modulus less than $\gamma$. In the nonautonomous case \eqref{1} such a set needs to be characterized dynamically by prescribing the growth rates of solutions with initial values in that set.

\begin{definition}[$\gamma$-exponentially stable set $M_\gamma$]\label{def:ExpStableSet}
For $\gamma \in \mathbb{R}$ the set
 \begin{equation*}
   M_{\gamma} \coloneqq \{x_0 \in \mathbb{R}^d \setminus\{0\} :
   \overline{\beta}_A(x_0) < \gamma
   \}
   \cup \{0\}
\end{equation*}
is called \emph{$\gamma$-exponentially stable set} of \eqref{1}.
\end{definition}

That $M_\gamma$ turns out to be a subspace for $\gamma \in \varrho_{\mathrm{B}}(A)$ and other important properties of $M_\gamma$, is the content of the following lemma:

\begin{lemma}[Properties of $M_\gamma$]\label{lem:gamma-prop}\hfill

(i) (Monotonicity) $M_\gamma$ is monotone
\begin{equation*}
   M_{\gamma_1} \subseteq M_{\gamma_2},
   \qquad
   \gamma_1 \leq \gamma_2,
\end{equation*}
and eventually constant
\begin{equation}\label{eq:Mgamma}
   M_\gamma
   =
   \begin{cases}
      \{0\}, & \text{ for }\gamma \in (-\infty, -\ln \|A^{-1}\|_{\infty}),
   \\
      \mathbb{R}^d, & \text{ for } \gamma \in (\ln \|A\|_{\infty}, \infty).
   \end{cases}
\end{equation}

(ii) ($M_\gamma$ is a subspace on resolvent intervals)
\begin{equation*}
   \gamma \in \varrho_{\mathrm{B}}(A)
   \quad\Rightarrow\quad
   M_{\gamma} \text{ is a linear subspace of } \mathbb{R}^{d}.
\end{equation*}

(iii) ($M_\gamma$ is constant on resolvent intervals) Let $\gamma_1, \gamma_2 \in \varrho_{\mathrm{B}}(A)$ with $\gamma_1 < \gamma_2$. Then exactly one of the following two alternatives holds and the statements in each alternative are equivalent:

\begin{tabular}{ll}
   \hspace*{7ex}Alternative I\hspace*{9ex} & \hspace*{12ex}Alternative II
\\
   \hspace*{3ex}(A) $[\gamma_1, \gamma_2] \subseteq \varrho_{\mathrm{B}}(A)$.
   &
   (A') There exists $\zeta \in (\gamma_1, \gamma_2) \cap \Sigma_{\mathrm{B}}(A)$.
\\
   \hspace*{3ex}(B) $M_{\gamma_1} = M_{\gamma_2}$.
   &
   (B') $\dim M_{\gamma_1} < \dim M_{\gamma_2}$.
\end{tabular}
\end{lemma}

\begin{proof}
(i) Follows with Remark \ref{rem:bohl-spectrum-bounded}.

(ii) Let \(\gamma \in \varrho_\mathrm B(A)\) and \(x_0,x_0' \in M_\gamma\), \(\alpha \in \R\).
If \(\alpha x_0 \neq 0\) then \(\overline\beta_A(\alpha x_0) = \overline\beta(x_0) < \gamma\) by Lemma \ref{lem:bohlproperties}(iv).
Thus \(\alpha x_0 \in M_\gamma\).
We have to show that \(x_0 + x_0' \in M_\gamma\).
This is clear if \(x_0 = 0\), \(x_0' = 0\) or \(x_0 + x_0' = 0\) so we only consider the case \(x_0, x_0', x_0+x_0' \neq 0\).
We show that \(\underline\beta_A(x_0 + x_0') < \gamma\).
Then \(\overline\beta_A(x_0 + x_0') < \gamma\) follows as otherwise we would have the contradiction \(\gamma \in [\underline\beta(x_0 + x_0'),\overline\beta_A(x_0 + x_0')] \subseteq \Sigma_\mathrm B(A)\).
Recall the definition of Lyapunov exponent (see e.g.\ \cite[p.\ 3]{Barreira2017})
\begin{equation*}
        \lambda_A(y) \coloneqq \limsup_{n\to\infty}\frac 1n\ln\Vert x(n,y)\Vert,
        \qquad y \in \mathbb \R^d\setminus\{0\}.
\end{equation*}
From the definition of Bohl and Lyapunov exponent it follows for \(y \in \mathbb \R^d\setminus\{0\}\) that
\begin{equation*}
        \lambda_A(y) \leq \overline\beta_A(y),
        \qquad\text{and}\qquad
        \underline\beta_A(y) \leq \lambda_A(y).
\end{equation*}
From \cite[p.\ 3]{Barreira2017} it follows that
\begin{equation*}
        \lambda_A(x_0+x_0') \leq \max\{\lambda_A(x_0), \lambda_A(x_0')\}.
\end{equation*}
Using these facts we conclude
\begin{align*}
        \underline\beta_A(x_0+x_0')
        \leq \lambda_A(x_0+x_0')
        & \leq \max\{\lambda_A(x_0), \lambda_A(x_0')\}
        \\
        & \leq \max\{\overline\beta_A(x_0), \overline\beta_A(x_0')\}
        < \gamma.
\end{align*}

(iii) $(A) \Rightarrow (B)$. That $M_{\gamma_1} \subseteq M_{\gamma_2}$ holds by (i). Let $x_0 \in M_{\gamma_2} \setminus M_{\gamma_1}$, i.e.
\begin{equation*}
   \overline{\beta}_A(x_0) \geq \gamma_1
   \quad \text{and} \quad
   \overline{\beta}_A(x_0) < \gamma_2
\end{equation*}
and consequently
\begin{equation*}
   [\underline{\beta}_A(x_0), \overline{\beta}_A(x_0)]
   \cap
   [\gamma_1, \gamma_2]
   \neq
   \emptyset,
\end{equation*}
contradicting $[\gamma_1, \gamma_2] \subseteq \varrho_{\mathrm{B}}(A)$.

$(B) \Rightarrow (A)$. Since $\gamma_2 \in \varrho_{\mathrm{B}}(A)$,
\begin{equation*}
   M_{\gamma_2} \cup \{x_0 \in \mathbb{R}^d \setminus \{0\} : \gamma_2 < \underline{\beta}_A(x_0)\}
   =
   \mathbb{R}^d.
\end{equation*}
Using the assumption $M_{\gamma_1} = M_{\gamma_2}$, it follows that
\begin{equation*}
   \overline{\beta}_A(x_0) < \gamma_1
   \quad \text{or} \quad
    \gamma_2 < \underline{\beta}_A(x_0)
   \qquad
   \text{for each }
   x_0 \in \mathbb{R}^d \setminus \{0\}.
\end{equation*}
As a consequence $[\gamma_1, \gamma_2] \subseteq \varrho_{\mathrm{B}}(A)$.

$(A')\Leftrightarrow (B')$. Obviously $(A')$ is the opposite of $(A)$. Using (i) and (ii), it follows that $(B')$ is the opposite of $(B)$.
\end{proof}

We are now in a position to formulate and prove the main result of this section.

\begin{theorem}[Bohl Spectral Theorem]
\label{T3}
The Bohl spectrum $\Sigma_{\mathrm{B}}(A)$
of system (\ref{1}) is the nonempty disjoint union of at most $d$ bounded intervals
\begin{equation*}
   \Sigma _{\mathrm{B}}(A)
   =
   I_1 \cup \dots \cup I_{\ell},
\end{equation*}
where $\ell \in \{1, \dots, d\}$, $\sup I_i \leq \inf I_{i+1}$ and $[\sup I_i, \inf I_{i+1}] \cap \varrho_{\mathrm B}(A) \neq \emptyset$ for $i \in \{1, \dots, \ell-1\}$ .

Moreover, setting $I_0 = I_{\ell+1} \coloneqq \emptyset$ and $\inf\emptyset \coloneqq +\infty$ and $\sup\emptyset \coloneqq -\infty$, there exists a corresponding filtration
\begin{equation*}
   \{0\}
   =
   V_{0} \subsetneq V_{1} \subsetneq \cdots \subsetneq V_{\ell }
   =
   \mathbb{R}^{d}
\end{equation*}
defined for $i \in \{0, \dots, \ell\}$ and $\gamma \in [\sup I_i, \inf I_{i+1}] \cap \varrho_{\mathrm B}(A)$ by
\begin{equation*}
   V_i
   \coloneqq
   \{ x_0 \in \mathbb{R}^{d} \setminus \{0\}
   :
   \overline{\beta}_A(x_0)
   <
   \gamma
   \}
   \cup \{0\}.
\end{equation*}
The definition of $V_i$ does not depend on the choice of $\gamma \in [\sup I_i, \inf I_{i+1}] \cap \varrho_{\mathrm B}(A)$.
\end{theorem}

\begin{proof}
%
Let $d_{0} < d_{1 }< \dots <d_{\ell}$ be the elements of the set
\begin{equation*}
   \{ \dim (M_{\gamma})
   :
   \gamma \in \varrho_{\mathrm{B}}(A)
   \}.
\end{equation*}
It is clear that $\ell \leq d$.
For $i\in \left\{ 0,...,\ell \right\} $, define%
\begin{equation*}
   J_{i}
   \coloneqq
   \{ \gamma \in \varrho_{\mathrm{B}}(A)
   :
   \dim (M_{\gamma}) = d_{i}
   \}
\end{equation*}
and note that $\varrho_{\mathrm{B}}(A) = J_0 \cup \dots \cup J_\ell$, where the union is disjoint.

We show that $J_{i}$ is an interval. To this end let $\gamma_1 < \gamma_2$ be two elements of $J_{i}$.
Since $\gamma_1, \gamma_2 \in \varrho_{\mathrm{B}}(A)$ and
$\dim (M_{\gamma_1}) = \dim (M_{\gamma_2})$, Lemma \ref{lem:gamma-prop}(iii) applies and yields Alternative I, proving that $[\gamma_1, \gamma_2] \subseteq J_{i}$.

Using \eqref{eq:Mgamma}, $d_0=0$, $d_\ell = d$, and $(-\infty, - \ln a) \subseteq J_0$, $(\ln a, \infty) \subseteq J_\ell$.
Therefore the complement $\Sigma_{\mathrm{B}}(A)$ of $\varrho_{\mathrm{B}}(A) = J_0 \cup \dots \cup J_\ell$ is the disjoint union of $\ell \in \{1, \dots, d\}$ bounded intervals
$I_1, \dots, I_\ell$ with $\sup I_i \leq \inf I_{i+1}$.

Now $[\sup I_i, \inf I_{i+1}] \cap \varrho_{\mathrm B}(A) = J_i \neq \emptyset$ for $i \in \{1,\dots,\ell-1\}$.
Using Alternative I of Lemma \ref{lem:gamma-prop}(iii), $V_i \coloneqq M_{\gamma_i}$ is well-defined for $\gamma_0 \in (-\infty, \inf I_1)$, $\gamma_i \in [\sup I_i, \inf I_{i+1}] \cap \varrho_{\mathrm B}(A)$ for $i \in \{1, \dots, \ell-1\}$, and $\gamma_\ell \in (\sup I_\ell, \infty)$ and satisfies $\{0\} = V_{0} \subsetneq V_{1} \subsetneq \cdots \subsetneq V_{\ell } = \mathbb{R}^{d}$.
\end{proof}

\section{Bohl dichotomy spectrum}

In this section we introduce a new spectrum based on the notion of Bohl dichotomy and prove a spectral theorem. Similarly as for the exponential dichotomy spectrum one associates to system \eqref{1} a parametrized family of nonautonomous difference equations which are exponentially weighted versions of \eqref{1} by considering for $\gamma \in \mathbb{R}$ the \emph{$\gamma$-shifted} system
\begin{equation}\label{gamma-shift}
   x(n+1) = \mathrm e^{-\gamma }A(n)x(n)
\end{equation}
Note that the transition matrix $\Phi_{\mathrm e^{-\gamma} A}$ of \eqref{gamma-shift} satisfies
\begin{equation*}
   \Phi_{\mathrm e^{-\gamma} A}(n,m)
   =
   \mathrm e^{-\gamma(n-m)}
   \Phi_A(n,m)
   \qquad \text{for }
   n, m \in \mathbb{N}.
\end{equation*}
For convenience we denote the solution $ \Phi_{\mathrm e^{-\gamma} A}(n,0)x_0$ of \eqref{gamma-shift} by $x_\gamma(\cdot,x_0)$. Then
\begin{equation*}
   x_\gamma(n,x_0)
   =
   \mathrm e^{-\gamma n} x(n,x_0)
   \qquad \text{for }
   n \in \mathbb{N}
\end{equation*}
with the solution $x(\cdot,x_0)$ of \eqref{1}.

To introduce a spectral notion based on the Bohl dichotomy we discuss whether \eqref{gamma-shift} admits a Bohl dichotomy, i.e.\ if there exist subspaces $L_1, L_2 \subseteq \mathbb{R}^d$ with $\mathbb{R}^{d}=L_{1} \oplus L_{2}$, $\alpha >0$, and functions $C_{1}, C_{2} \colon \mathbb{R}^{d}\rightarrow \left( 0,\infty \right) $ such
that
\begin{align}\label{Dich-gamma1}
   \| x_\gamma(n, x_0) \|
   &\leq
   C_1(x_0) \mathrm e^{-\alpha (n-m)} \| x_\gamma(m, x_0) \|,
   \quad x_0 \in L_1,
   n \geq m,
\\ \label{Dich-gamma2}
   \| x_\gamma(n, x_0) \|
   &\geq
   C_2(x_0) \mathrm e^{\alpha (n-m)} \| x_\gamma(m, x_0) \|,
   \quad x_0 \in L_2, n \geq m.
\end{align}

\begin{definition}[Bohl dichotomy spectrum]
\label{DB}The Bohl dichotomy spectrum of \eqref{1} is defined as
\begin{equation*}
   \Sigma _{\mathrm{BD}}(A)
   \coloneqq
   \left\{ \gamma \in \mathbb{R}:x(n+1)=\mathrm e^{-\gamma }A(n)x(n)%
    \text{ has no Bohl dichotomy}\right\} .
\end{equation*}
Its complement $\varrho_{\mathrm{BD}}(A) \coloneqq \mathbb{R} \setminus \Sigma _{\mathrm{BD}}(A)$ is called the resolvent of \eqref{1}.
\end{definition}

\begin{remark}[Bohl dichotomy spectrum is bounded]\label{rem:spectrum-bounded}
\begin{equation*}
   \Sigma _{\mathrm{BD}}(A) \subseteq
   \big[ - \ln \|A^{-1}\|_{\infty}, \ln \|A\|_{\infty}\big].
\end{equation*}
This follows from the fact that $A \in \mathcal{L}^{\mathrm{Lya}}(\mathbb{N},\mathbb{R}^{d\times d})$ is a Lyapunov sequence, and the following estimate for \(m,n\in\mathbb N\) with \(n\geq m\) and \(x_0\in\R^d\),
\begin{align*}
    \Vert x_\gamma(n,x_0)\Vert
    &= \Vert\mathrm e^{-\gamma(n-m)}\Phi_A(n,m)\mathrm e^{-\gamma m}\Phi_A(m,0)x_0\Vert
    \\
    &\leq \mathrm e^{-\gamma(n-m)}\mathrm e^{(n-m)\ln \|A\|_{\infty}}\Vert x_\gamma(m,x_0)\Vert
\end{align*}
proving that for $\gamma >\ln \|A\|_{\infty}$, the $\gamma$-shifted system \eqref{gamma-shift} has a Bohl dichotomy with $L_1 = \R^d$ and $L_2=\{0\}$.
Similarly the estimate for $m \geq n$
\begin{align*}
    \Vert x_\gamma(m,x_0)\Vert
    &= \Vert \mathrm e^{\gamma(n-m)}\Phi_A(m,n)\mathrm e^{-\gamma n}\Phi_A(n,0)x_0\Vert
    \\
    &\leq \mathrm e^{\gamma(n-m)}\mathrm e^{(n-m) \ln\Vert A^{-1}\Vert_\infty}\Vert x_\gamma(n,x_0)\Vert.
\end{align*}
shows that for $\gamma < -\ln\Vert A^{-1}\Vert_\infty$, system \eqref{gamma-shift} has a Bohl dichotomy with $L_1 = \{0\}$ and $L_2=\R^d$.
\end{remark}

\begin{remark}[Bohl dichotomy resolvent is open]\label{rem:resolvent-open}
The resolvent $\varrho_{\mathrm{BD}}(A)$ is open, since for any $\gamma \in \varrho_{\mathrm{BD}}(A)$ the estimates
\begin{align*}
   \| x(n, x_0) \|
   &\leq
   C_1(x_0) \mathrm e^{(\gamma-\alpha) (n-m)} \| x(m, x_0) \|,
   \quad x_0 \in L_1,
   n \geq m,
\\
   \| x(n, x_0) \|
   &\geq
   C_2(x_0) \mathrm e^{(\gamma + \alpha) (n-m)} \| x(m, x_0) \|,
   \quad x_0 \in L_2, n \geq m,
\end{align*}
hold and for $\varepsilon \coloneqq \alpha / 2$ and $\zeta \in (\gamma - \varepsilon, \gamma + \varepsilon)$, using the fact that $\zeta - \varepsilon > \gamma - \alpha$ and $\zeta + \varepsilon < \gamma + \alpha$, they imply
\begin{align*}
   \| x(n, x_0) \|
   &\leq
   C_1(x_0) \mathrm e^{(\zeta-\varepsilon) (n-m)} \| x(m, x_0) \|,
   \quad x_0 \in L_1,
   n \geq m,
\\
   \| x(n, x_0) \|
   &\geq
   C_2(x_0) \mathrm e^{(\zeta + \varepsilon) (n-m)} \| x(m, x_0) \|,
   \quad x_0 \in L_2, n \geq m.
\end{align*}
\end{remark}

Similarly as for the Bohl spectrum in Definition \ref{def:ExpStableSet} in the last section, we introduce a family of sets of initial conditions of solutions with prescribed asymptotic behavior.

\begin{definition}[$\gamma$-attractive subset $S_\gamma$]
For $\gamma \in \mathbb{R}$ the set
 \begin{equation*}
   S_{\gamma} \coloneqq \{x_0 \in \mathbb{R}^d :
   \lim_{n\to\infty}x_\gamma(n,x_0) = 0
   \}
\end{equation*}
is called \emph{$\gamma$-attractive subset} of \eqref{1}.
\end{definition}

The sets $S_\gamma$ and their properties will play a crucial role in the formulation and proof of the Bohl dichotomy spectral theorem below (cp.\ also Lemma \ref{lem:gamma-prop}).

\begin{lemma}[Properties of $S_\gamma$]\label{lem:propSgamma}\hfill

(i) (Subspace) $S_\gamma$ is a linear subspace of $\mathbb{R}^{d}$ for each $\gamma \in \mathbb{R}$.

(ii) (Bohl dichotomy space) $S_\gamma$ is the Bohl dichotomy space $L_1$ on resolvent intervals.
\begin{equation*}
   \left.
   \begin{gathered}
      \gamma \in \varrho_{\mathrm{BD}}(A)
      \,\text{and } x(n+1)=\mathrm{e}^{-\gamma }A(n)x(n)
   \\
      \text{has a Bohl dichotomy on } L_1 \oplus L_2 = \mathbb{R}^d
   \end{gathered}
   \right\}
   \quad\Rightarrow\quad
   L_1 = S_{\gamma}.
\end{equation*}

(iii) (Monotonicity) $S_\gamma$ is monotone
\begin{equation*}
   S_{\gamma_1} \subseteq S_{\gamma_2},
   \qquad
   \gamma_1 \leq \gamma_2,
\end{equation*}
and eventually constant
\begin{equation*}
   S_\gamma
   =
   \begin{cases}
      \{0\}, & \text{ for } \gamma \in (-\infty, -\ln \|A^{-1}\|_{\infty}),
   \\
      \mathbb{R}^d, & \text{ for } \gamma \in (\ln \|A\|_{\infty}, \infty).
   \end{cases}
\end{equation*}

(iv) ($S_\gamma$ is constant on resolvent intervals) Let $\gamma_1, \gamma_2 \in \varrho_{\mathrm{BD}}(A)$ with $\gamma_1 < \gamma_2$. Then exactly one of the following two alternatives holds and the statements in each alternative are equivalent:

\begin{tabular}{ll}
   \hspace*{7ex}Alternative I\hspace*{9ex} & \hspace*{12ex}Alternative II
\\
   \hspace*{3ex}(A) $[\gamma_1, \gamma_2] \subseteq \varrho_{\mathrm{BD}}(A)$.
   &
   (A') There exists $\zeta \in (\gamma_1, \gamma_2) \cap \Sigma_{\mathrm{BD}}(A)$.
\\
   \hspace*{3ex}(B) $S_{\gamma_1} = S_{\gamma_2}$.
   &
   (B') $\dim S_{\gamma_1} < \dim S_{\gamma_2}$.
\end{tabular}
\end{lemma}

\begin{proof}
(i) This follows from the fact that \eqref{1} is a linear equation.

(ii) Suppose that $\gamma \in \varrho_{\mathrm{B}}(A)$ and system \eqref{gamma-shift} has a Bohl dichotomy \eqref{Dich-gamma1}, \eqref{Dich-gamma2}, on a splitting $L_1 \oplus L_2 = \mathbb{R}^d$ w.r.t.\ $\gamma$.
The inclusion $L_1\subseteq S_\gamma$ is clear.
To show that also $S_\gamma \subseteq L_1$, let $x_0\notin L_1$.
We show that $x_\gamma(\cdot,x_0)$ is not a null-sequence.
Let $x_0 = x_1 + x_2$, with $x_1\in L_1$ and $x_2\in L_2\setminus\{0\}$.
Then by \eqref{Dich-gamma2}
\begin{align*}
    \Vert x_\gamma(n,x_0)\Vert
    = \Vert x_\gamma(n,x_1) + x_\gamma(n,x_2)\Vert
    &\geq \big\vert \Vert x_\gamma(n,x_1)\Vert - \Vert x_\gamma(n,x_2)\Vert\big\vert
    \\
    &\geq \big\vert \Vert x_\gamma(n,x_1)\Vert - C_2(x_2)\mathrm e^{\alpha n}\Vert x_2\Vert\big\vert.
\end{align*}
The right-hand side tends to infinity, for $n$ to infinity and the assertion follows. In particular, if \eqref{gamma-shift} has a Bohl dichotomy on a splitting $L_1 \oplus L_2 = \mathbb{R}^d$ then $L_1$ is unique.

(iii) This is a consequence of Remark \ref{rem:spectrum-bounded} and (ii).

(iv) $(A) \Rightarrow (B)$. Assume that $S_{\gamma_1} \neq S_{\gamma_2}$ and define
\begin{equation*}
   \zeta_0
   \coloneqq
   \sup \{\zeta \in [\gamma_1, \gamma_2] : S_{\zeta} = S_{\gamma_1}\}
   \in \varrho_{\mathrm{BD}}(A).
\end{equation*}
By Remark \ref{rem:resolvent-open} there exists $\varepsilon > 0$ such that $S_\zeta = S_{\zeta_0}$ for $\zeta \in (\zeta_0 - \varepsilon, \zeta_0 + \varepsilon)$ which contradicts the definition of $\zeta_0$.

$(B) \Rightarrow (A)$. For $\gamma_1 \in \varrho_{\mathrm{BD}}(A)$ the first dichotomy estimate
\begin{equation*}
   \| x(n, x_0) \|
   \leq
   C_1(x_0) \mathrm e^{(\gamma_1-\alpha_1) (n-m)} \| x(m, x_0) \|,
   \quad x_0 \in L_1,
   n \geq m,
\end{equation*}
holds with $\alpha_1 > 0$ on $L_1$, and by (ii) $L_1 = S_{\gamma_1}$. For $\gamma_2 \in \varrho_{\mathrm{BD}}(A)$ the second dichotomy estimate
\begin{equation*}
   \| x(n, x_0) \|
   \geq
   C_2(x_0) \mathrm e^{(\gamma_2 + \alpha_2) (n-m)} \| x(m, x_0) \|,
   \quad x_0 \in L_2, n \geq m,
\end{equation*}
holds with an $\alpha_2 > 0$ on a subspace $L_2$ which by (ii) has the property that $S_{\gamma_2} \oplus L_2 = \mathbb{R}^d$. Since $S_{\gamma_2} = S_{\gamma_1} = L_1$, we get $L_1 \oplus L_2 = \mathbb{R}^d$. With
$\alpha \coloneqq \min\{\alpha_1, \alpha_2\}$ it follows that
\begin{align*}
   \| x(n, x_0) \|
   &\leq
   C_1(x_0) \mathrm e^{(\gamma_1-\alpha) (n-m)} \| x(m, x_0) \|,
   \quad x_0 \in L_1,
   n \geq m,
\\
   \| x(n, x_0) \|
   &\geq
   C_2(x_0) \mathrm e^{(\gamma_2 + \alpha) (n-m)} \| x(m, x_0) \|,
   \quad x_0 \in L_2, n \geq m.
\end{align*}
As a consequence also $\gamma \in \varrho_{\mathrm{BD}}(A)$ for each $\gamma \in [\gamma_1, \gamma_2]$.

$(A') \Leftrightarrow (B')$. Obviously $(A')$ is the opposite of $(A)$. Using the fact that $S_{\gamma_1} \subseteq S_{\gamma_2}$, also $(B')$ is the opposite of $(B)$. Since $(A) \Leftrightarrow (B)$, also $(A') \Leftrightarrow (B')$.
\end{proof}

\begin{remark}[Bohl Dichotomy Subspaces]
(a) Suppose that system \eqref{1} has a Bohl dichotomy on a decomposition \(L_1\oplus L_2\) and \(\widehat L_1\oplus\widehat L_2\) of \(\R^d\).
Then by Lemma \ref{lem:propSgamma}(ii) \(L_1 = S_\gamma = \widehat L_1\).

(b) If system \eqref{1} has a Bohl dichotomy on a decomposition \(L_1\oplus L_2\), then system \eqref{1} has a Bohl dichotomy on any decomposition of the form \(L_1\oplus\widehat L_2\) of \(\R^d\).
We prove this fact in Lemma \ref{lem:BD-characterization}$(ii)\Leftrightarrow(iii)$ where \(L_1\) is explicitly defined and \(L_2\) is chosen arbitrarily complementary to \(L_1\).
\end{remark}

The following theorem is the main result of this section.

\begin{theorem}[Bohl Dichotomy Spectral Theorem]\label{thm:Bohl-Dichotomy-Spectrum}
The Bohl dichotomy spectrum $\Sigma _{\mathrm{BD}}(A)$
of system (\ref{1}) is the nonempty disjoint union of at most $d$ compact intervals
\begin{equation*}
   \Sigma _{\mathrm{BD}}(A)
   =
   [\alpha _{1}, \beta _{1}] \cup \dots \cup [\alpha_{\ell },\beta _{\ell}],
\end{equation*}
where $\alpha_1 \leq \beta_1 < \alpha_2 \leq \beta_2 < \dots < \alpha_\ell \leq \beta_\ell$ and $\ell \in \{1, \dots, d\}$.

Moreover, there exists a corresponding filtration
\begin{equation}
   \{0\}
   =
   V_{0} \subsetneq V_{1} \subsetneq \cdots \subsetneq V_{\ell }
   =
   \mathbb{R}^{d},
\end{equation}
satisfying the following characterization for $i \in \{0, \dots, \ell\}$
\begin{equation*}
   V_i
   =
   \{ x_0 \in \mathbb{R}^{d}
   :
   \lim_{n \to \infty} x(n,x_0) e^{-\gamma n}
   =
   0
   \}
   \quad
   \text{for each }
   \gamma \in (\beta_i, \alpha_{i+1}),
\end{equation*}
with $\beta_0 \coloneqq -\infty$ and $\alpha_{\ell + 1} \coloneqq \infty$.
\end{theorem}

\begin{proof}
For $k \in \{0,\dots,d\}$ the sets $\{\gamma \in \varrho_{\mathrm{BD}}(A) : \dim S_\gamma = k\}$ are intervals by Lemma \ref{lem:propSgamma}(iv), open by Remark \ref{rem:resolvent-open}, disjoint by definition, and for $k=0$ and $k=d$ unbounded to the left and right, respectively, by Remark \ref{rem:spectrum-bounded}. Since
\begin{equation*}
   \varrho_{\mathrm{BD}}(A)
   =
   \bigcup_{k = 0}^{d}
   \{\gamma \in \varrho_{\mathrm{BD}}(A) : \dim S_\gamma = k\},
\end{equation*}
its complement $\Sigma_{\mathrm{BD}}(A)$ is the disjoint union of $\ell \in \{1, \dots, d\}$ closed intervals $[\alpha_1, \beta_1], \dots, [\alpha_\ell, \beta_\ell]$ with
$\alpha_1 \leq \beta_1 < \alpha_2 \leq \beta_2 < \dots < \alpha_\ell \leq \beta_\ell$ and $\ell \in \{1, \dots, d\}$. For $i \in \{1, \dots, \ell-1\}$ there exists a $k \in \{0,\dots,d\}$ with
\begin{equation*}
   (\beta_i, \alpha_{i+1})
   =
   \{\gamma \in \varrho_{\mathrm{BD}}(A) : \dim S_\gamma = k\}.
\end{equation*}
For $\gamma_1, \gamma_2 \in (\beta_i, \alpha_{i+1})$ with $\gamma_1 < \gamma_2$, the fact that $\dim S_{\gamma_1} = \dim S_{\gamma_2}$ and $S_{\gamma_1} \subseteq S_{\gamma_2}$ implies that $S_{\gamma_1} = S_{\gamma_2}$, proving that
\begin{equation*}
   V_i
   \coloneqq
   S_\gamma
   =
   \{ x_0 \in \mathbb{R}^{d}
   :
   \lim_{n \to \infty} x(n,x_0) \mathrm e^{-\gamma n}
   =
   0
   \}
\end{equation*}
is well-defined for $\gamma \in (\beta_i, \alpha_{i+1})$.
\end{proof}

\section{Relation between the Bohl, Bohl dichotomy and exponential dichotomy spectra}

We recall the notion of exponential dichotomy spectrum, some of its properties and the exponential dichotomy spectral theorem from \cite{AulbachSiegmund2002, Poetzsche2010} with slightly adjusted notation.

\begin{definition}[Exponential dichotomy spectrum]
\label{ED-spectrum}The exponential dichotomy spectrum of \eqref{1} is defined as
\begin{equation*}
   \Sigma _{\mathrm{ED}}(A)
   \coloneqq
   \left\{ \gamma \in \mathbb{R}:x(n+1)=\mathrm e^{-\gamma }A(n)x(n)%
    \text{ has no exponential dichotomy}\right\} .
\end{equation*}
Its complement $\varrho_{\mathrm{ED}}(A) \coloneqq \mathbb{R} \setminus \Sigma _{\mathrm{ED}}(A)$ is called the resolvent of \eqref{1}.
\end{definition}

\begin{theorem}[Exponential Dichotomy Spectral Theorem]\label{thm:Exponential-Dichotomy-Spectrum}
The exponential dichotomy spectrum $\Sigma _{\mathrm{ED}}(A)$
of system (\ref{1}) is the nonempty disjoint union of at most $d$ compact intervals
\begin{equation*}
   \Sigma _{\mathrm{ED}}(A)
   =
   [\alpha _{1}, \beta _{1}] \cup \dots \cup [\alpha_{\ell },\beta _{\ell}],
\end{equation*}
where $\alpha_1 \leq \beta_1 < \alpha_2 \leq \beta_2 < \dots < \alpha_\ell \leq \beta_\ell$ and $\ell \in \{1, \dots, d\}$.

Moreover, there exists a corresponding filtration
\begin{equation}
   \{0\}
   =
   V_{0} \subsetneq V_{1} \subsetneq \cdots \subsetneq V_{\ell }
   =
   \mathbb{R}^{d}
\end{equation}
satisfying the following characterization for $i \in \{0, \dots, \ell\}$
\begin{equation*}
   V_i
   =
   \{ x_0 \in \mathbb{R}^{d}
   :
   \lim_{n \to \infty} x(n,x_0) \mathrm e^{-\gamma n}
   =
   0
   \}
   \quad
   \text{for each }
   \gamma \in (\beta_i, \alpha_{i+1}),
\end{equation*}
with $\beta_0 \coloneqq -\infty$ and $\alpha_{\ell + 1} \coloneqq \infty$.
\end{theorem}

\begin{proof}
See e.g.\ \cite{AulbachSiegmund2002, Poetzsche2010, Russ2017}.
\end{proof}

\begin{remark}[Minimum and maximum of $\Sigma _{\mathrm{ED}}(A)$ are Bohl exponents]\label{rem:MinMaxED}
\begin{equation*}
   \min \Sigma _{\mathrm{ED}}(A) = \underline{\beta}_A(\mathbb{R}^d)
   \qquad
   \text{and}
   \qquad
   \max \Sigma _{\mathrm{ED}}(A) = \overline{\beta}_A(\mathbb{R}^d).
\end{equation*}
To show that
$\max \Sigma _{\mathrm{ED}}(A) \geq \overline{\beta}_A(\mathbb{R}^d)$, let $\gamma > \max \Sigma _{\mathrm{ED}}(A)$. Then
\begin{equation*}
   \lim_{n \to \infty} x(n,x_0) \mathrm e^{-\gamma n}
   =
   0
   \qquad
   \text{for each }
   x_0 \in \mathbb{R}^d
\end{equation*}
by Theorem \ref{thm:Exponential-Dichotomy-Spectrum}.
Using the fact that $\gamma \in \rho_{\mathrm{ED}}(A)$, it follows that there exists $K > 0$, $\alpha > 0$, such that
\begin{equation*}
   \| x(n, x_0) \|
   \leq
   K \mathrm e^{(\gamma-\alpha) (n-m)} \| x(m, x_0) \|,
   \qquad x_0 \in \mathbb{R}^d,
   n \geq m.
\end{equation*}
By Lemma \ref{lem:bohlproperties}(iii), $\gamma \geq \overline{\beta}_A(\mathbb{R}^d)$.
To show that
$\max \Sigma _{\mathrm{ED}}(A) \leq \overline{\beta}_A(\mathbb{R}^d)$, let $\gamma > \overline{\beta}_A(\mathbb{R}^d)$ choose $\alpha \in (0, \gamma - \overline{\beta}_A(\mathbb{R}^d))$. Then $\gamma - \alpha > \overline{\beta}_A(\mathbb{R}^d)$ and
again by Lemma \ref{lem:bohlproperties}(iii) there exists a $K > 0$ such that
\begin{equation*}
   \| x(n, x_0) \|
   \leq
   K \mathrm e^{(\gamma - \alpha)(n-m)} \| x(m, x_0) \|,
   \qquad x_0 \in \mathbb{R}^d,
   n \geq m,
\end{equation*}
i.e.\ $\gamma \in \rho_{\mathrm{ED}}(A)$ and by Theorem \ref{thm:Exponential-Dichotomy-Spectrum}, also $\gamma > \max \Sigma _{\mathrm{ED}}(A)$, proving that $\max \Sigma _{\mathrm{ED}}(A) = \overline{\beta}_A(\mathbb{R}^d)$.

The equality $\min \Sigma _{\mathrm{ED}}(A) = \underline{\beta}_A(\mathbb{R}^d)$ follows similarly.
\end{remark}

\begin{remark}[Exponential dichotomy spectrum for scalar and diagonal systems]\label{rem:ED_scalar}
(a)
If $d=1$ then system \eqref{1} is of the form $x(n+1)=a(n)x(n)$, $n \in \mathbb{N}$, and then by Theorem \ref{thm:Exponential-Dichotomy-Spectrum} and Remark \ref{rem:MinMaxED}
\begin{equation*}
    \Sigma _{\mathrm{ED}}(a)
    =
    \big[\underline{\beta}_a(\mathbb{R}), \overline{\beta}_a(\mathbb{R}) \big]
\end{equation*}
with
\begin{equation*}
   \underline{\beta}_a(\mathbb{R})
   =
   \liminf_{n-m \to \infty} \tfrac{1}{n-m} \ln \prod_{k=m}^{n-1} |a(k)|
   \quad\text{and}\quad
   \overline{\beta}_a(\mathbb{R})
   =
   \limsup_{n-m \to \infty} \tfrac{1}{n-m} \ln \prod_{k=m}^{n-1} |a(k)|.
\end{equation*}

(b)
If system \ref{1} is diagonal, i.e.\ $A = \operatorname{diag}(a_{11},\dots,a_{dd})$ then
\begin{equation*}
    \Sigma_{\mathrm{ED}}(A) = \bigcup_{k=1}^d\Sigma_{\mathrm{ED}}(a_{kk}).
\end{equation*}
For a proof see, e.g.\ \cite[Corollary 3.25]{Poetzsche2012}.
\end{remark}

To discuss the relation between the Bohl spectrum, Bohl dichotomy spectrum and exponential dichotomy spectrum, we show two preparatory lemmas on characterizations of Bohl dichotomy and exponential dichotomy.

\begin{lemma}[Characterization of Bohl dichotomy]\label{lem:BD-characterization}
The following three statements are equivalent:

(i) System \eqref{1} has a Bohl dichotomy.

(ii) There exists a splitting $L_1 \oplus L_2 = \mathbb{R}^d$ with
\begin{equation*}
   \sup_{x_0 \in L_1 \setminus \{0\}}
   \overline{\beta}_{A}(x_0)
   < 0
   \qquad \text{and} \qquad
   \inf_{x_0 \in L_2 \setminus \{0\}}
   \underline{\beta}_{A}(x_0)
   > 0.
\end{equation*}

(iii) There is \(\alpha>0\), such that for all \(x_0\in\R^d\setminus\{0\}\),
\begin{equation*}
    \overline\beta_A(x_0)\leq-\alpha\qquad\text{or}\qquad\underline\beta_A(x_0)\geq\alpha.
\end{equation*}

\end{lemma}

\begin{proof} $(i) \Leftrightarrow (ii)$:
Suppose that system \eqref{1} has a Bohl dichotomy.
Let \(L_1,L_2\subseteq\R^d\) be such that the inequalities \eqref{Dich-gamma1} resp.\ \eqref{Dich-gamma2} hold on \(L_1\) resp.\ \(L_2\) for \(\alpha>0\).
Then from Lemma \ref{lem:bohlproperties}(iii) it follows for \(x_0\in L_1\setminus\{0\}\) for which the inequality \eqref{Dich-gamma1} holds, that \(-\alpha\geq\overline\beta_A(x_0)\) and for \(x_0\in L_2\setminus\{0\}\) for which the inequality \eqref{Dich-gamma2} holds, that \(\alpha\leq\underline\beta_A(x_0)\).

For the converse, let \(L_1\) and \(L_2\) be the subspaces, for which the inequality for the exponents hold.
Then there is \(\alpha>0\), such that
\begin{equation*}
    \sup_{x_0 \in L_1 \setminus \{0\}}
   \overline{\beta}_{A}(x_0)
   <-\alpha<0<\alpha<
   \inf_{x_0 \in L_2 \setminus \{0\}}
   \underline{\beta}_{A}(x_0).
\end{equation*}
For \(x_0\in L_1\setminus\{0\}\) it follows from \(\overline\beta_A(x_0)<-\alpha\) and Lemma \ref{lem:bohlproperties}(iii), that there is \(C(x_0)>0\), such that for \(x_0\) the inequality \eqref{Dich-gamma1} holds.
Similarly, the inequality \eqref{Dich-gamma2} for \(x_0 \in L_2 \setminus\{0\}\) can be shown.

$(i) \Leftrightarrow (iii)$:
Suppose that system \(\eqref{1}\) has a Bohl dichotomy.
By $(ii)$ there are subspaces \(L_1,L_2\subseteq\R^d\) and \(\alpha>0\), such that for \(x_1\in L_1\setminus\{0\}\), we have \(\overline\beta_A(x_1)\leq-\alpha\) and if \(x_2\in L_2\setminus\{0\}\), we have \(\underline\beta_A(x_2)\geq\alpha\).
Now let \(x_0=x_1+x_2\in\R^d\setminus\{0\}\) with \(x_1\in L_1\) and \(x_2\in L_2\).
In case $x_2 = 0$, we have \(\overline\beta_A(x_0) = \overline\beta_A(x_1)\leq-\alpha\).
In case $x_1 = 0$, we have \(\underline\beta_A(x_0) = \underline\beta_A(x_2)\geq \alpha\).
Otherwise $x_1 \neq 0$ and $x_2 \neq 0$ and we conclude, using Lemma \ref{lem:bohlproperties}(vi),
\begin{equation*}
    \underline\beta_A(x_0) = \underline\beta_A(x_1+x_2) \geq \underline\beta_A(x_2)\geq\alpha > 0.
\end{equation*}
For the converse, suppose there is \(\alpha>0\), such that for all \(x_0\in\R^d\setminus\{0\}\),
\begin{equation*}
    \overline\beta_A(x_0)\leq-\alpha\quad\text{or}\quad\underline\beta_A(x_0)\geq\alpha.
\end{equation*}
We define
\begin{equation*}
    L_1
    \coloneqq
    \big\{x_0\in\R^d\setminus\{0\} : \overline\beta_A(x_0)\leq-\alpha\big\}\cup\{0\},
\end{equation*}
and show that $L_1$ is a subspace of $\mathbb{R}^d$.
To this end let $x_1,x_2\in L_1$ and $\lambda\in\mathbb R$.
From Lemma \ref{lem:bohlproperties}(iv) it follows that $\overline\beta_A(\lambda x_1) = \overline\beta_A(x_1)\leq -\alpha$ i.e.\ $\lambda x_1\in L_1$.
By Lemma \ref{lem:bohlproperties}(v) it follows from \(\overline\beta_A(x_1),\overline\beta_A(x_2)<0\) that \(\underline\beta_A(x_1+x_2)\leq 0\).
By assumption it follows that \(\overline\beta_A(x_1+x_2)\leq-\alpha\) and we conclude that $x_1+x_2\in L_1$.
Hence $L_1$ is a subspace.

Now let $L_2\subseteq\mathbb R^d$ be any subspace such that $L_1\oplus L_2 =\mathbb R^d$.
If $x_0\in L_2$, then either $x_0 = 0$ or $x_0 \not\in  L_1$, i.e.\ $\overline\beta_A(x_0)>-\alpha$. By assumption it then follows that $\underline\beta_A(x_0)\geq\alpha$, hence
\begin{equation*}
   \sup_{x_0 \in L_1 \setminus \{0\}}
   \overline{\beta}_{A}(x_0)
   < 0
   \qquad \text{and} \qquad
   \inf_{x_0 \in L_2 \setminus \{0\}}
   \underline{\beta}_{A}(x_0)
   > 0
\end{equation*}
and $(ii)$ holds, which is equivalent to $(i)$.
\end{proof}

\begin{lemma}[Characterization of exponential dichotomy]\label{lem:ED-characterization}
The following statements are equivalent:

(i) System (\ref{1}) has an exponential dichotomy.

(ii) There exists a splitting $L_1 \oplus L_2 = \mathbb{R}^d$ with
\begin{equation*}
   \overline{\beta}_{A}(L_1)
   < 0
   \qquad \text{and} \qquad
   \underline{\beta}_{A}(L_2)
   > 0.
\end{equation*}
\end{lemma}

\begin{proof}
The proof is very similar to that of Lemma \ref{lem:BD-characterization}, so we omit the details.
\end{proof}

The following two theorems show that the Bohl dichotomy spectrum is the closure of the Bohl spectrum, as well as contained in the exponential dichotomy spectrum.

\begin{theorem}[Bohl dichotomy spectrum is closure of Bohl spectrum]\label{thm:cl_B_eq_BD}
It holds that
\begin{equation*}
    \operatorname{cl}\Sigma_B(A)
    =
    \Sigma_{\mathrm{BD}}(A).
\end{equation*}
\end{theorem}

\begin{proof}
\(\operatorname{cl}\Sigma_{\mathrm B}(A)\subseteq\Sigma_{\mathrm{BD}}(A)\):
We show that \(\Sigma_B(A)\subseteq\Sigma_{\mathrm{BD}}(A)\).
To this end, we show that \(\rho_{\mathrm{BD}}(A)\subseteq\rho_{\mathrm B}(A)\).
Let \(\gamma\in\rho_{\mathrm{BD}}(A)\).
Then by Lemma \ref{lem:BD-characterization} there exists a splitting $L_1 \oplus L_2 = \mathbb{R}^d$, such that for \(x_0 = x_1 + x_2 \in\mathbb R^d\setminus\{0\}\) with \(x_1\in L_1\) and \(x_2\in L_2\) it holds that
\begin{equation*}
   \overline{\beta}_{\mathrm{e}^{-\gamma}A}(x_1)
   < 0
   \qquad \text{and} \qquad
   \underline{\beta}_{\mathrm{e}^{-\gamma}A}(x_2)
   > 0.
\end{equation*}
Using Lemma \ref{lem:bohlproperties}(vi) for $x(n+1) = \mathrm{e}^{-\gamma}A(n)x(n)$, it follows that
\begin{align*}
   &&\overline{\beta}_{A}(x_0)
   < \gamma,&&\text{if \(x_2 = 0\)},
   \\
   &&\underline{\beta}_{A}(x_0)
   = \underline{\beta}_{A}(x_1 + x_2)
   \geq \underline{\beta}_{A}(x_2)
   >\gamma,&&\text{if \(x_2 \neq 0\)}.
\end{align*}
Consequently \(\gamma\notin\big[\underline\beta_A(x_0),\overline\beta_A(x_0)\big]\) for all \(x_0\in\mathbb R^d\setminus\{0\}\) and hence \(\gamma\in\rho_{\mathrm{B}}(A)\), proving that \(\Sigma_B(A)\subseteq\Sigma_{\mathrm{BD}}(A)\).
Since \(\Sigma_{\mathrm{BD}}(A)\) is closed by Theorem \ref{thm:Bohl-Dichotomy-Spectrum}, the inclusion \(\operatorname{cl}\Sigma_B(A)\subseteq\Sigma_{\mathrm{BD}}(A)\) follows.

\(\operatorname{cl}\Sigma_\mathrm{B}(A)\supseteq\Sigma_\mathrm{BD}(A)\):
Let \(\gamma\in\Sigma_\mathrm{BD}(A)\).
To show that \(\gamma\in\operatorname{cl}\Sigma_\mathrm{B}(A)\) we equivalently show that
\begin{equation*}
	\alpha\coloneqq\inf\{\vert\gamma-\beta\vert : \beta\in\Sigma_\mathrm{B}(A)\} = 0.
\end{equation*}
Assume to the contrary that \(\alpha>0\).
Then \(\gamma\in\rho_{\mathrm B}(A)\).
We will apply Lemma \ref{lem:BD-characterization}.
To this end let \(x_0\in\mathbb R^d\setminus\{0\}\).
Since \(\big[\underline\beta_A(x_0),\overline\beta_A(x_0)\big]\subseteq\Sigma_\mathrm{B}(A)\), it follows that either
\begin{equation*}
   (i)\quad\gamma<\underline\beta_A(x_0)
   \qquad\text{or}\qquad
   (ii)\quad\gamma>\overline\beta_A(x_0).
\end{equation*}
It follows by definition of \(\alpha\) in case (i), that \(\alpha<\underline\beta_A(x_0)-\gamma = \underline\beta_{\mathrm e^{-\gamma}A}(x_0)\) and in case (ii) that \(\alpha<\gamma-\overline\beta_A(x_0) = -\overline\beta_{\mathrm e^{-\gamma}A}(x_0)\).
By Lemma \ref{lem:BD-characterization} it follows that \(\gamma\in\rho_{\mathrm{BD}(A)}\) which is a contradiction to \(\gamma\in\Sigma_\mathrm{BD}(A)\).
\end{proof}

\begin{theorem}[Exponential dichotomy spectrum contains Bohl dichotomy spectrum]\label{thm:BE_subeq_ED}
It holds that
\begin{equation*}
    \Sigma_{\mathrm{BD}}(A)
    \subseteq
    \Sigma_{\mathrm{ED}}(A).
\end{equation*}
\end{theorem}

\begin{proof}
From the definition of both spectra, it easily follows that \(\rho_{\mathrm{ED}}(A)\subseteq\rho_{\mathrm{BD}}(A)\).
\end{proof}

We introduce a notion of transformation between difference equations of the form \eqref{1} and show that the spectra are preserved under transformations.

\begin{definition}[Dynamic equivalence]\label{def:dynamic-equivalence}
Let $A, B \in \mathcal{L}^{\mathrm{Lya}}(\mathbb{N},\mathbb{R}^{d\times d})$. The two systems
\begin{equation}\label{two-systems}
      x(n+1) = A(n) x(n)
      \quad \text{and} \quad
      y(n+1) = B(n) y(n),
      \qquad
      n \in \mathbb{N},
\end{equation}
are called \emph{dynamically equivalent}, if there exists $T \in \mathcal{L}^{\mathrm{Lya}}(\mathbb{N},\mathbb{R}^{d\times d})$ with
\begin{equation*}
   B(n)
   =
   T(n+1)^{-1} A(n) T(n),
   \quad
   n \in \mathbb{N}.
\end{equation*}
$T$ is called \emph{Lyapunov transformation} between the two systems \eqref{two-systems}. The two systems \eqref{two-systems} are said to be \emph{dynamically equivalent via $T$}.
\end{definition}

\begin{remark}[Relation of solutions of dynamically equivalent systems]\label{rem:dynamic-equivalence}
Using the fact that
\begin{equation*}
   T(n) \Phi_B(n,m)
   =
   \Phi_A(n,m) T(m)
   , \qquad n, m \in \mathbb{N},
\end{equation*}
it follows for $x_0, y_0 \in \mathbb{R}^d$ that
\begin{equation*}
      x(n, x_0)
      =
      T(n) y(n, T(0)^{-1} x_0)
      \quad \text{and} \quad
      y(n, y_0)
      =
      T(n)^{-1} x(n, T(0) y_0),
      \quad
      n \in \mathbb{N}.
\end{equation*}
\end{remark}

\begin{lemma}[Invariance of Bohl exponents under dynamic equivalence]\label{lem:bohl-invariance}
Let $L \subseteq \mathbb{R}^d$. If the two system \eqref{two-systems} are dynamically equivalent via $T$, then
\begin{equation*}
   \overline{\beta}_A(L) = \overline{\beta}_B(T(0)^{-1} L)
   \qquad \text{and} \qquad
   \underline{\beta}_A(L) = \underline{\beta}_B(T(0)^{-1} L).
\end{equation*}
\end{lemma}

\begin{proof}
By Remark \ref{rem:dynamic-equivalence},
\begin{equation*}
    \Phi_B(n,0)
    =
    T(n)^{-1}\Phi_A(n,0)T(0)
    , \qquad
    n \in \mathbb{N}.
\end{equation*}
For \(y_0\in T(0)^{-1}L\setminus\{0\}\), and \(m,n\in\N\) with \(n>m\) we compute
\begin{align*}
    \frac{1}{n-m}\ln\frac{\Vert\Phi_B(n,0)y_0\Vert}{\Vert\Phi_B(m,0)y_0\Vert}
    &=
    \frac{1}{n-m}\ln\frac{\Vert T(n)^{-1}\Phi_A(n,0)T(0)y_0\Vert}{\Vert T(m)^{-1}\Phi_A(m,0)T(0)y_0\Vert}
    \\
    &\leq
    \frac{1}{n-m}\ln\frac{\Vert T(n)^{-1}\Vert\cdot\Vert T(m)\Vert\cdot\Vert \Phi_A(n,0)T(0)y_0\Vert}{\Vert \Phi_A(m,0)T(0)y_0\Vert}
    \\
    &\leq
    \frac{\ln\big(\Vert T^{-1}\Vert_\infty\cdot\Vert T\Vert_\infty\big)}{n-m}+\frac{1}{n-m}\ln\frac{\Vert \Phi_A(n,0)T(0)y_0\Vert}{\Vert \Phi_A(m,0)T(0)y_0\Vert}.
\end{align*}
Hence for \(N\in\N\), it holds that
\begin{align*}
    &\sup_{n-m>N}\sup_{y_0\in T(0)^{-1}L\setminus\{0\}}\frac{1}{n-m}\ln\frac{\Vert\Phi_B(n,0)y_0\Vert}{\Vert\Phi_B(m,0)y_0\Vert}
    \\
    &\qquad\leq
    \sup_{n-m>N}\Bigg(\frac{\ln\big(\Vert T^{-1}\Vert_\infty\cdot\Vert T\Vert_\infty\big)}{n-m}
    +
    \sup_{x_0\in L\setminus\{0\}}\frac{1}{n-m}\ln\frac{\Vert \Phi_A(n,0)x_0\Vert}{\Vert \Phi_A(m,0)x_0\Vert}\Bigg)
    \\
    &\qquad\leq
    \frac{\ln\big(\Vert T^{-1}\Vert_\infty\cdot\Vert T\Vert_\infty\big)}{N+1}
    +\sup_{n-m>N} \sup_{x_0\in L\setminus\{0\}}\frac{1}{n-m}\ln\frac{\Vert \Phi_A(n,0)x_0\Vert}{\Vert \Phi_A(m,0)x_0\Vert}.
\end{align*}
Letting \(N\) tend to infinity, it follows that
\begin{equation*}
   \overline\beta_B(T(0)^{-1}L)
   \leq
   \overline\beta_A(L).
\end{equation*}
Since $A(n) = T(n+1) B(n) T(n)^{-1}$, $n \in \mathbb{N}$, it also follows that
\begin{equation*}
   \overline\beta_A(L)
   \leq
   \overline\beta_B(T(0)^{-1}L),
\end{equation*}
proving that $\overline\beta_A(L) = \overline\beta_B(T(0)^{-1}L)$. The equality $\underline{\beta}_A(L) = \underline{\beta}_B(T(0)^{-1} L)$ follows similarly.
\end{proof}

\begin{theorem}[Invariance of spectra under dynamic equivalence]\label{thm:spectra-invariance}
If the two systems \eqref{two-systems} are dynamically equivalent then
\begin{equation*}
   \Sigma_{\mathrm{B}}(A) = \Sigma_{\mathrm{B}}(B)
   , \qquad
   \Sigma_{\mathrm{BD}}(A) = \Sigma_{\mathrm{BD}}(B)
   , \qquad
   \Sigma_{\mathrm{ED}}(A) = \Sigma_{\mathrm{ED}}(B).
\end{equation*}
\end{theorem}

\begin{proof}
Let $T$ be the Lyapunov transformation between the two systems \eqref{two-systems}.

$\Sigma_{\mathrm{B}}(A) = \Sigma_{\mathrm{B}}(B)$. This follows from Lemma \ref{lem:bohl-invariance} and the fact that $T(0)$ is bijective, since
\begin{align*}
   \bigcup_{x_0 \in \mathbb{R}^d \setminus \{0\}}
   \big[ \underline{\beta}_A(x_0), \overline{\beta}_A(x_0) \big]
   &=
   \bigcup_{x_0 \in \mathbb{R}^d \setminus \{0\}}
   \big[ \underline{\beta}_B(T(0)^{-1} x_0), \overline{\beta}_B(T(0)^{-1} x_0) \big]
\\
   &=
   \bigcup_{y_0 \in \mathbb{R}^d \setminus \{0\}}
   \big[ \underline{\beta}_B(y_0), \overline{\beta}_B(y_0) \big].
\end{align*}

$\Sigma_{\mathrm{BD}}(A) = \Sigma_{\mathrm{BD}}(B)$.
We show that $\rho_{\mathrm{BD}}(A) = \rho_{\mathrm{BD}}(B)$, using Lemmas \ref{lem:bohl-invariance} and \ref{lem:BD-characterization}.
To this end let $\gamma \in \rho_{\mathrm{BD}}(A)$. By Lemma \ref{lem:BD-characterization}(ii), there are subspaces $L_1,L_2$ with $L_1 \oplus L_2 = \mathbb R^d$ and
\begin{equation*}
   \sup_{x_0 \in L_1 \setminus \{0\}}
   \overline{\beta}_{\mathrm e^{-\gamma}A}(x_0)
   < 0
   \qquad \text{and} \qquad
   \inf_{x_0 \in L_2 \setminus \{0\}}
   \underline{\beta}_{\mathrm e^{-\gamma}A}(x_0)
   > 0.
\end{equation*}
Since $T(0)$ is bijective and linear, we have $T(0)^{-1}L_1 \oplus T(0)^{-1}L_2 = \mathbb R^d$.
Moreover, the two systems
\begin{equation*}
      x(n+1) = \mathrm e^{-\gamma}A(n) x(n)
      \quad \text{and} \quad
      y(n+1) = \mathrm e^{-\gamma}B(n) y(n),
      \qquad
      n \in \mathbb{N},
\end{equation*}
are dynamically equivalent via $T$, since $\mathrm e^{-\gamma}B(n) = T(n+1)^{-1} \mathrm e^{-\gamma} A(n) T(n)$, $n \in \mathbb{N}$.
By Lemma \ref{lem:bohl-invariance}, we conclude
\begin{align*}
   \sup_{y_0 \in T(0)^{-1}L_1 \setminus \{0\}}
   \overline{\beta}_{\mathrm e^{-\gamma}B}(y_0)
   &=
   \sup_{x_0 \in L_1 \setminus \{0\}}
   \overline{\beta}_{\mathrm e^{-\gamma}A}(x_0)
   < 0,
   \\
   \inf_{y_0 \in T(0)^{-1}L_2 \setminus \{0\}}
   \underline{\beta}_{\mathrm e^{-\gamma}B}(y_0)
   &=
   \inf_{x_0 \in L_2 \setminus \{0\}}
   \underline{\beta}_{\mathrm e^{-\gamma}A}(x_0)
   > 0.
\end{align*}
Hence $\gamma \in \rho_{\mathrm{BD}}(B)$ by Lemma \ref{lem:BD-characterization}(ii), that is $\rho_{\mathrm{BD}}(A) \subseteq \rho_{\mathrm{BD}}(B)$.
Similarly one can show that $\rho_{\mathrm{BD}}(A) \supseteq \rho_{\mathrm{BD}}(B)$.

$\Sigma_{\mathrm{ED}}(A) = \Sigma_{\mathrm{ED}}(B)$.
This follows similarly as $\Sigma_{\mathrm{BD}}(A) = \Sigma_{\mathrm{BD}}(B)$ using Lemma \ref{lem:ED-characterization}.
\end{proof}

We now transform system \eqref{1} into upper triangular form $A = (a_{ij})_{i, j = 1, \dots, d}$, $a_{ij} = 0$ for $i > j$, which by Theorem \ref{thm:spectra-invariance} has the same Bohl, Bohl dichotomy and exponential dichotomy spectra, respectively. We then compare its spectra with the spectra of its diagonal part $x(n+1) = A_{\operatorname{diag}}(n) x(n)$ with $A_{\operatorname{diag}} \coloneqq \operatorname{diag}(a_{11}, \dots, a_{dd})$.

\begin{theorem}[Upper triangular normal form]\label{thm:UTnormal-form}
Let \(A\in\mathcal L^{\mathrm{Lya}}(\mathbb N,\mathbb R^{d\times d})\).
Then there is \(B\in\mathcal L^{\mathrm{Lya}}(\mathbb N,\mathbb R^{d\times d})\), such that \(B(n)\) is upper triangular for \(n\in\mathbb N\) and such that the systems
\begin{equation*}
      x(n+1) = A(n) x(n)
      \quad \text{and} \quad
      y(n+1) = B(n) y(n),
      \qquad
      n \in \mathbb{N},
\end{equation*}
are dynamically equivalent via \(T\in\mathcal L^{\mathrm{Lya}}\), whereby \(T(n)\) is in the special orthogonal group for \(n\in\N\).
\end{theorem}

\begin{proof}
For the proof, see \cite[p.\ 52, Theorem 3.2.1]{Barreira2017}.
\end{proof}

Together with Theorems \ref{thm:cl_B_eq_BD} and \ref{thm:BE_subeq_ED}, the following theorem concludes the discussion of general relations between the Bohl, Bohl dichotomy and exponential dichotomy spectrum.

\begin{theorem}[Spectra of upper triangular systems]\label{thm:SpecUT}
Assume that system \eqref{1} is upper triangular. Then
\begin{center}
   \begin{tabular}{ccccc}
      $\Sigma_{\mathrm{B}}(A)$
      & $\subseteq$
      & $\Sigma_{\mathrm{BD}}(A)$
      & $\subseteq$ &
      $\Sigma_{\mathrm{ED}}(A)$
   \\
       \rotatebox[origin=c]{270}{$\subseteq$}
       & &
       \rotatebox[origin=c]{270}{$\subseteq$}
       & &
       \rotatebox[origin=c]{90}{$=$}
   \\
      $\Sigma_{\mathrm{B}}(A_{\mathrm{diag}})$
      & $=$
      & $\Sigma_{\mathrm{BD}}(A_{\mathrm{diag}})$
      & $=$ &
      $\Sigma_{\mathrm{ED}}(A_{\mathrm{diag}})$
   \end{tabular}
\end{center}
\end{theorem}

\begin{proof}
That $\Sigma_{\mathrm{B}}(A) \subseteq \Sigma_{\mathrm{BD}}(A) \subseteq \Sigma_{\mathrm{ED}}(A)$ resp.\ $\Sigma_{\mathrm{B}}(A_{\mathrm{diag}}) \subseteq \Sigma_{\mathrm{BD}}(A_{\mathrm{diag}}) \subseteq \Sigma_{\mathrm{ED}}(A_{\mathrm{diag}})$ has been shown in Theorem \ref{thm:cl_B_eq_BD} and Theorem \ref{thm:BE_subeq_ED}.

$\Sigma_{\mathrm{ED}}(A_{\mathrm{diag}})\subseteq\Sigma_{\mathrm{B}}(A_{\mathrm{diag}})$:
By Remark \ref{rem:ED_scalar} it follows that $\Sigma_{\mathrm{ED}}(A_{\mathrm{diag}})$ is the union of the Bohl intervals
\begin{equation*}
    \big[\underline\beta_{A_{\mathrm{diag}}}(e_k),\overline\beta_{A_{\mathrm{diag}}}(e_k)\big],
\end{equation*}
whereby \(e_1,\dots,e_d\) is the standard basis of \(\R^d\).

$\Sigma_{\mathrm{ED}}(A) = \Sigma_{\mathrm{ED}}(A_{\mathrm{diag}})$:
For a proof, see e.g.\ \cite{Poetzsche2012} Corollary 3.25.

$\Sigma_{\mathrm{B}}(A) \subseteq \Sigma_{\mathrm{B}}(A_{\mathrm{diag}})$: This follows from $\Sigma_{\mathrm{B}}(A) \subseteq \Sigma_{\mathrm{BD}}(A) \subseteq \Sigma_{\mathrm{ED}}(A) = \Sigma_{\mathrm{B}}(A_{\mathrm{diag}})$.

$\Sigma_{\mathrm{BD}}(A) \subseteq \Sigma_{\mathrm{BD}}(A_{\mathrm{diag}})$:
Follows from $\Sigma_{\mathrm{BD}}(A) \subseteq \Sigma_{\mathrm{ED}}(A) = \Sigma_{\mathrm{BD}}(A_{\mathrm{diag}})$.
\end{proof}

In the following remark we show that most inclusions in Theorem \ref{thm:SpecUT} might be strict inclusions.

\begin{remark}[Diagonal significance]
The significance of the diagonal entries of an upper triangular matrix function $A$ for the spectrum is an important question when it comes to the computation of the spectrum. In \cite{Czornik2019} an example $A$ is constructed for which
\begin{equation*}
    \sup_{x_0\in\R^2\setminus\{0\}}\overline\beta_A(x_0)<0,
    \quad\text{and}\quad
    \overline{\beta}_A(\mathbb R^2)>0.
\end{equation*}
The following relations result from the above inequalities
\begin{center}
   \begin{tabular}{ccccc}
      $\Sigma_{\mathrm{B}}(A)$
      & $\subseteq$
      & $\Sigma_{\mathrm{BD}}(A)$
      & $\subsetneq$ &
      $\Sigma_{\mathrm{ED}}(A)$
   \\
       \rotatebox[origin=c]{270}{$\subsetneq$}
       & &
       \rotatebox[origin=c]{270}{$\subsetneq$}
       & &
       \rotatebox[origin=c]{90}{$=$}
   \\
      $\Sigma_{\mathrm{B}}(A_{\mathrm{diag}})$
      & $=$
      & $\Sigma_{\mathrm{BD}}(A_{\mathrm{diag}})$
      & $=$ &
      $\Sigma_{\mathrm{ED}}(A_{\mathrm{diag}})$
   \end{tabular}
\end{center}
In fact, since
\begin{equation*}
   \sup_{x_0\in\R^2\setminus\{0\}}\overline\beta_A(x_0)
   =
   \sup \Sigma_{\mathrm{BD}}(A)
   \qquad\text{and}\qquad
   \overline{\beta}_A(\mathbb R^2)
   =
   \sup \Sigma_{\mathrm{ED}}(A) ,
\end{equation*}
it follows that $\Sigma_{\mathrm{BD}}(A) \subsetneq \Sigma_{\mathrm{ED}}(A)$.
In Theorem \ref{thm:SpecUT} we have seen that
\begin{equation*}
    \Sigma_{\mathrm{ED}}(A) = \Sigma_{\mathrm{B}}(A_{\mathrm{diag}}) = \Sigma_{\mathrm{BD}}(A_{\mathrm{diag}}) = \Sigma_{\mathrm{ED}}(A_{\mathrm{diag}}).
\end{equation*}
Using $\Sigma_{\mathrm{B}}(A) \subseteq \Sigma_{\mathrm{BD}}(A)$, we conclude that $\Sigma_{\mathrm{B}}(A) \subsetneq \Sigma_{\mathrm{B}}(A_{\mathrm{diag}})$ and $\Sigma_{\mathrm{BD}}(A) \subsetneq \Sigma_{\mathrm{BD}}(A_{\mathrm{diag}})$.
It is an open question whether there exists a system \eqref{1} such that $\Sigma_{\mathrm{B}}(A) \subsetneq \Sigma_{\mathrm{BD}}(A)$.
\end{remark}

\section*{Declarations}

\subsection*{Ethical Approval}

This work does involve neither human nor animal studies.

\subsection*{Competing interests}

There are no financial or personal interests undermining good scientific conduct and the integrity of the article.

\subsection*{Authors' contributions}

All authors have contributed equally to the article.

\subsection*{Funding}

The research of A.\ Czornik was supported by the Polish National Agency for Academic Exchange according to the decision PPN/BEK/2020/1/00188/UO/00001.

\subsection*{Availability of data and materials}

For the references used in the article publication place and time are given in the reference section.
No datasets have been used in the article.


\end{document}